\documentclass[11pt,a4paper]{amsart}
\usepackage[T1]{fontenc}
\usepackage[utf8]{inputenc}
\usepackage{lmodern,amssymb}
\usepackage{graphicx}
\usepackage[hidelinks]{hyperref}
\numberwithin{equation}{section}
\newtheorem{theorem}{Theorem}[section]
\newtheorem{lemma}[theorem]{Lemma}
\newtheorem{proposition}[theorem]{Proposition}
\theoremstyle{remark}

\newcommand{\R}{\mathbb R}
\newcommand{\C}{\mathbb C}
\newcommand{\PW}{\mathcal{PW}}
\newcommand{\sgn}{\operatorname{sgn}}
\newcommand{\supp}{\operatorname{supp}}
\newcommand{\norm}[1]{\left\|#1\right\|}
\title[Approximation and Bohr--Favard estimates on the sphere]{On best approximation and estimates of Bohr--Favard type on the sphere for functions with bounded Laplacian}
\author{Alexey Solyanik}
\date{October 10, 2026}
\begin{document}
\begin{abstract}
We determine the best constants in estimates of the $L^\infty$ norm of functions on the sphere with bounded Laplacian whose first spherical harmonic components vanish. In every dimension $d\geq2$, these constants are expressed through weighted $L^1$ extremal polynomials.
On $S^2$ the constant is an explicit positive series, and the best coefficient of $n^{-2}$ is $7\zeta(3)/4$. On $S^3$ the best coefficient of $n^{-2}$ is a positive integral involving the H\"ormander--Bernhardsson extremal function. In higher dimensions the rescaled extremal polynomials converge to weighted entire extremals, which give asymptotic formulas for the Bohr constants on $S^4$ and $S^6$.

The most unexpected and, in our view, interesting observation is that, for every fixed $n\geq2$ on $S^2$ and $S^3$, the constants in best approximation estimates are \emph{strictly smaller} than the corresponding Bohr--Favard constants.
\end{abstract}
\maketitle

\section{The Bohr and Favard theorems}\label{sec:classicalBohrFavard}

Sharp estimates for the uniform norm of periodic functions with missing low frequencies were first obtained by Harald Bohr in 1935 ($r=1$)~\cite{Bohr} and Jean Favard in 1936 ($r\geq2$)~\cite{Favard1936}. For a $2\pi$-periodic function $f$ whose $r$th derivative is bounded in $L^\infty$ by one and for which $\widehat f(k)=0$ for $|k|<n$, the inequality takes the form
\begin{equation}\label{eq:BohrFavard}
 \norm{f}_{L^\infty}\leq\frac{K_r}{n^r}.
\end{equation}
Here $r$ and $n$ are positive integers, and $f^{(r-1)}$ is absolutely continuous. The Favard constants are
\begin{equation}\label{eq:FavardConstants}
 K_r=\frac4\pi\sum_{j\in4\mathbb Z+1}\frac1{j^{r+1}}.
\end{equation}
In particular, $K_1=\pi/2$ and $K_2=\pi^2/8$.

In his subsequent work, Favard~\cite{Favard1937} obtained a sharp estimate for the best approximation of functions with a bounded $r$th derivative by trigonometric polynomials. If $f$ is $2\pi$-periodic, $f^{(r-1)}$ is absolutely continuous, and its $r$th derivative is bounded in $L^\infty$ by one, then
\begin{equation}\label{eq:FavardApproximation}
 E_{n-1}(f)_{L^\infty}\leq\frac{K_r}{n^r}.
\end{equation}
Here $T_{n-1}$ denotes the space of trigonometric polynomials of degree at most $n-1$, and
\[
 E_{n-1}(f)_X=\inf_{t\in T_{n-1}}\norm{f-t}_X.
\]
The assumption that the $r$th derivative is bounded by one entails no loss of generality, by rescaling.

Thus the same constants occur in both estimates (see  DeVore and Lorentz~\cite[Chapter~7, \S\,4]{DeVoreLorentz}).

\emph{Both estimates~\eqref{eq:BohrFavard} and~\eqref{eq:FavardApproximation} are sharp for every fixed $n$.} Their sharpness is witnessed by the Euler spline (see DeVore and Lorentz~\cite[Chapter~7, \S\,4]{DeVoreLorentz}).

It is important to distinguish this from the weaker assertion that $K_r$ is the best constant valid simultaneously for all $n$.

It may seem that estimates~\eqref{eq:BohrFavard} and~\eqref{eq:FavardApproximation} are equivalent. One implication does hold: by a special case of the Hahn--Banach theorem proved by Helly in 1912~\cite{Helly}, estimate~\eqref{eq:FavardApproximation} follows from~\eqref{eq:BohrFavard} (see Section~\ref{sec:favardconstants} for the corresponding duality argument on the sphere).

In the opposite direction, however, estimate~\eqref{eq:BohrFavard} does not follow formally from~\eqref{eq:FavardApproximation}, even though the two estimates appear closely related.

Of course, estimate~\eqref{eq:FavardApproximation} can also be proved directly. For this, we look for an operator $B_r$ which inverts the $r$th derivative on functions of mean zero:
\[
 f=B_r\bigl(f^{(r)}\bigr).
\]
The assumption of zero mean entails no loss of generality: subtracting a constant changes neither the derivative nor the best approximation error. Such an operator is given by convolution with the Bernoulli kernel. With the notation and normalization of~\cite{Solyanik},
\[
 \mathcal B_r(\theta)=\sum_{k\ne0}\frac{e^{ik\theta}}{(ik)^r},
 \qquad f=\mathcal B_r*f^{(r)}.
\]
Here convolution and the $L^1$ norm are taken with respect to the normalized measure $d\theta/(2\pi)$. Subtracting a trigonometric polynomial from the kernel immediately gives
\[
 E_{n-1}(f)_{L^\infty}
 \leq\norm{\mathcal B_r-t_{n-1}}_{L^1}
 \norm{f^{(r)}}_{L^\infty},
\]
where $t_{n-1}$ is any trigonometric polynomial of degree at most $n-1$.

Thus the problem reduces to best $L^1$ approximation of the Bernoulli kernel. Favard~\cite{Favard1936} obtained the exact formula
\begin{equation}\label{eq:BernoulliBestApproximation}
 E_{n-1}(\mathcal B_r)_{L^1}=\frac{K_r}{n^r}.
\end{equation}
See also DeVore and Lorentz~\cite[Chapter~7, \S\,4]{DeVoreLorentz}.

Thus, to obtain Favard-type estimates by this method, one needs the best $L^1$ approximation of the kernel of the corresponding inverse operator (here, the inverse of $r$th differentiation).

As an example, consider the results of~\cite{Solyanik} on approximation of Lipschitz functions on $[-1,1]$ by \emph{algebraic polynomials}. After the substitution $x=\cos\theta$, the problem leads to $L^1$ approximation of the kernels of an inverse of the operator
\[
 D=\frac1{\sin\theta}\,\partial_\theta.
\]
This is an ordinary differential operator with a singular coefficient. Here the functions and the right-hand sides are even and periodic. With the constant of integration chosen to give zero mean, its inverse is the operator $K$ from~\cite[formula~(6.7)]{Solyanik}:
\begin{equation}\label{eq:KryakinInverse}
 (Kh)(\theta)=\sin\theta\,(h*\mathcal K_1)(\theta)-\cos\theta\,(h*\mathcal K_2)(\theta),
\end{equation}
where
\[
 \mathcal K_1(\theta)=\mathcal B_1(\theta)\cos\theta,
 \qquad
 \mathcal K_2(\theta)=\mathcal B_1(\theta)\sin\theta.
\]
In~\cite[Theorem~1]{Solyanik}, the following exact values were obtained for every integer $n\geq2$:
\begin{equation}\label{eq:KryakinBestApproximation}
 E_{n-1}(\mathcal K_1)_{L^1}=\tan\frac{\pi}{2n},
 \qquad
 E_{n-1}(\mathcal K_2)_{L^1}=\sec\frac{\pi}{2n}-1.
\end{equation}
These identities yield Nikol'skii-type estimates for approximation of Lipschitz functions on $[-1,1]$ by \emph{algebraic polynomials}; see~\cite[Theorem~2]{Solyanik}.

To conclude this brief review, we note that the best $L^1$ approximants to the kernels give explicit trigonometric polynomials realizing the corresponding approximation estimates for functions.

\section{Constructive approximation of functions on the sphere}\label{sec:sphericalApproximation}

The approximation of functions on spheres has a long history; see Nikol'skii and Lizorkin~\cite{NikolskiiLizorkin} and Dai and Xu~\cite{DaiXu} for results and further references. Approximation estimates and related methods were developed in the works of Kushnirenko~\cite{Kushnirenko}, Newman and Shapiro~\cite{NewmanShapiro}, Ragozin~\cite{Ragozin}, Bavinck~\cite{Bavinck}, and Pawelke~\cite{Pawelke}.

Let $d\geq2$, and let $S^d\subset\R^{d+1}$ be the unit sphere of dimension $d$, equipped with normalized surface measure $d\sigma_d$. The Laplace--Beltrami operator $\Delta=\Delta_{S^d}=\operatorname{div}_{S^d}\nabla_{S^d}$ can be computed by applying the Euclidean Laplacian to the extension that is constant along rays:
\[
 \Delta_{S^d}f(\omega)=\Delta_{\R^{d+1}}\widetilde f(\omega),
 \qquad \widetilde f(x)=f\!\left(\frac{x}{|x|}\right).
\]
Its eigenfunctions are the spherical harmonics. Those of degree $\ell$ form the space $\mathcal H_\ell$ of restrictions to the sphere of homogeneous harmonic polynomials of degree $\ell$ on $\R^{d+1}$, and
\[
 -\Delta_{S^d}Y=\ell(\ell+d-1)Y,
 \qquad Y\in\mathcal H_\ell,
\]
for each nonnegative integer $\ell$. We shall therefore approximate functions by finite sums of spherical harmonics. For an integer $n\geq1$, let $\mathcal H_{<n}$ denote the space of sums of spherical harmonics of degrees $0,\ldots,n-1$. We write $f\perp\mathcal H_{<n}$ if all these components of $f$ vanish. Thus $n$ is the first allowed degree, and $E_{n-1}(f)_{L^\infty(S^d)}$ denotes the best uniform approximation by spherical harmonics of degrees at most $n-1$.

For background on spherical harmonics and zonal kernels, see Stein and Weiss~\cite[Chapter~IV, \S\,2]{SteinWeiss}. An excellent survey of spherical functions and constructive approximation on the sphere is given by Freeden, Michel, and Simons~\cite{FreedenMichelSimons}.

For a smooth function $f$ whose Laplacian is bounded in $L^\infty$ by one, the classical estimate for best uniform approximation takes the form
\begin{equation}\label{eq:sphericalJacksonOrder}
 E_{n-1}(f)_{L^\infty(S^d)}\leq\frac{C_d}{n^2}.
\end{equation}
Pawelke~\cite[Theorem~3.3]{Pawelke} proves this estimate in arbitrary dimension and, on p.~477, attributes the case of $S^2$ to Kushnirenko's work of 1959~\cite{Kushnirenko}.

For the same class of functions, under the additional condition $f\perp\mathcal H_{<n}$, Kamzolov~\cite{Kamzolov} established the corresponding Bohr--Favard estimate
\begin{equation}\label{eq:sphericalBohrOrder}
 \norm{f}_{L^\infty(S^d)}\leq\frac{C'_d}{n^2}.
\end{equation}
His result applies more generally to compact symmetric spaces of rank one. The constants in these two estimates depend only on $d$, and the estimates hold for every positive integer $n$.

In the present work we seek the sharp constants in both estimates, restricting ourselves to $r=2$, that is, to the Laplace--Beltrami operator. To state the two problems precisely, define, for each positive integer $n$,
\begin{equation}\label{eq:sphericalBohrFavard}
 H(n,S^d)=\sup\left\{\norm{f}_\infty:\; f\perp\mathcal H_{<n},\; \norm{\Delta f}_\infty\leq1\right\},
\end{equation}
\begin{equation}\label{eq:sphericalFavardApproximation}
 K(n,S^d)=\sup\left\{E_{n-1}(f)_{L^\infty(S^d)}:\; \norm{\Delta f}_\infty\leq1\right\}.
\end{equation}
The suprema are taken over smooth complex-valued functions on $S^d$. The same constants are obtained for functions with bounded weak Laplacian, by heat regularization. Here $n$ specifies the harmonic cutoff, and the dimension $d\geq2$ is arbitrary. The letter $H$ refers to Harald Bohr, and $K$ to Jean Favard. The best coefficients in the corresponding estimates of order $n^{-2}$, valid simultaneously for all positive integers $n$, are $H(S^d)=\sup_{n\geq1}n^2H(n,S^d)$ and $K(S^d)=\sup_{n\geq1}n^2K(n,S^d)$.

With some latitude in terminology, we shall call the determination of $H(n,S^d)$ and $H(S^d)$ the \emph{Bohr problem}, and the determination of $K(n,S^d)$ and $K(S^d)$ the \emph{Favard problem}. This terminology distinguishes the two types of estimates, although Favard also proved the classical norm estimate for $r\geq2$.

We begin with the Green representation, which explains why approximation of an inverse kernel enters both questions. On $S^2$ the best kernel approximant comes from trigonometric interpolation at equally spaced angles and gives an explicit formula for $H(n,S^2)$. Its best uniform coefficient is $H(S^2)=7\zeta(3)/4$. On $S^3$ the interpolation nodes are instead determined by a weighted $L^1$ extremal problem; rescaling that problem leads to the H\"ormander--Bernhardsson function and the constant $H(S^3)$. An exact transfer argument then establishes the bound for every $n$, not only asymptotically.

We also compare best approximation with linear approximation. On $S^2$ the kernel method is optimal among all linear methods, yet $K(n,S^2)<H(n,S^2)$ for each $n\geq2$. Section~\ref{sec:favardconstants} gives the exact dual formulation and extends the strict inequality to $S^3$. Thus the coincidence of the two sharp constants in the classical periodic problem does not persist on these spheres. The exact values of $K(n,S^d)$ for $n\geq2$, and a possible separation of the constants uniform in $n$, remain open here.

Section~\ref{sec:BohrAllDimensions} extends the kernel interpolation argument and the exact formula for the Bohr constant at each $n$ to arbitrary dimension. Section~\ref{sec:weightedEntireLimit} identifies the limiting polynomial problem with a weighted entire-function problem, of which the H\"ormander--Bernhardsson problem is the case $d=3$. The pole structure of the even-dimensional Green kernels then yields the limits of $n^2H(n,S^4)$ and $n^2H(n,S^6)$ in Section~\ref{sec:evenBohrLimits}. Section~\ref{sec:evenBohrUniform} shows that these limits are also the best coefficients valid for every $n$.

\section{Green representation for the Bohr and Favard problems}\label{sec:greenrepresentation}

As in the periodic argument, the first step is to recover a function from its image under the differential operator. We work with functions of mean zero. This is already required by the missing-harmonic condition in the Bohr problem; in the approximation problem, subtracting the mean changes neither the Laplacian nor the error of best approximation.

The inverse of $-\Delta$ on these functions is rotationally invariant and is represented by a real zonal Green kernel $G_d$. Here zonal means that the kernel depends only on the scalar product of its two arguments. With $g=-\Delta f$, the representation is
\begin{equation}\label{eq:GreenRepresentation}
 f(x)=\int_{S^d}G_d(x\cdot y)g(y)\,d\sigma_d(y).
\end{equation}
We choose the kernel to have mean zero. Convolution with it annihilates constants and multiplies a spherical harmonic of positive degree $\ell$ by $1/(\ell(\ell+d-1))$, the reciprocal eigenvalue of $-\Delta$. These multipliers give~\eqref{eq:GreenRepresentation} by the spherical harmonic expansion. We shall verify the explicit kernels on $S^2$ and $S^3$ below, and give a representation valid in every dimension in Section~\ref{sec:BohrAllDimensions}.

Let $p$ be an algebraic polynomial of degree at most $n-1$. Replacing $G_d$ by $p$ in~\eqref{eq:GreenRepresentation} gives a spherical polynomial $t_{n-1}$ in $\mathcal H_{<n}$, since $p(x\cdot y)$ has degree at most $n-1$ in $x$. Subtracting the two representations and using rotational invariance yields, for any fixed pole $e$,
\begin{equation}\label{eq:kernelApproximationBound}
 \norm{f-t_{n-1}}_\infty
 \leq\norm{g}_\infty\int_{S^d}|G_d(e\cdot y)-p(e\cdot y)|\,d\sigma_d(y).
\end{equation}
This bounds the best approximation error of $f$. If $f\perp\mathcal H_{<n}$, then $g$ has the same missing harmonics, and $t_{n-1}=0$ because $p(x\cdot y)$ also has degree at most $n-1$ in $y$. The very same integral therefore bounds the norm of $f$ in the Bohr problem.

This explains the next task: find the best $L^1$ approximation of the Green kernel. To prove sharpness in the Bohr problem, we shall use the sign of the kernel's approximation error as $g$, and verify that it has the required missing harmonics. Sharpness for best approximation of functions is a separate question, to which we return beginning with Section~\ref{sec:favardS2}.

\section{The Bohr constant in arbitrary dimension}\label{sec:BohrAllDimensions}

We first determine the sharp Bohr constant at each fixed $n$ in every dimension. The Green representation reduces the upper bound to polynomial approximation of a zonal kernel. To prove sharpness, we choose the interpolation nodes so that the sign of the error has precisely the required vanishing moments. The resulting sign function will be the negative Laplacian of a Bohr extremal.

Let $d\geq2$. For a fixed pole $e\in S^d$, normalized surface measure induces, in the variable $t=e\cdot\omega$, the probability measure
\begin{equation}\label{eq:generalZonalMeasure}
 d\mu_d(t)=c_d(1-t^2)^{(d-2)/2}\,dt,
 \qquad c_d=\frac{\Gamma((d+1)/2)}{\sqrt\pi\,\Gamma(d/2)}.
\end{equation}
To obtain a representation of $G_d$ valid in every dimension, use the Poisson kernel of the unit ball, in this normalization:
\[
 \mathcal P_{\rho,d}(t)=\frac{1-\rho^2}{(1-2\rho t+\rho^2)^{(d+1)/2}},
 \qquad 0<\rho<1.
\]
Convolution with this kernel preserves constants and multiplies a spherical harmonic of degree $\ell$ by $\rho^\ell$; this is the Poisson extension of its homogeneous harmonic polynomial; see~\cite{SteinWeiss}. Consequently the mean-zero Green kernel has the representation
\begin{equation}\label{eq:generalGreenPoisson}
 G_d(t)=\frac1{d-1}\int_0^1\frac{1-\rho^{d-1}}{\rho}
 \bigl(\mathcal P_{\rho,d}(t)-1\bigr)\,d\rho.
\end{equation}
Indeed, the integral converges in $L^1(\mu_d)$: near zero, $\mathcal P_{\rho,d}-1=O(\rho)$ uniformly in $t$, while its $L^1(\mu_d)$ norm is at most two for all $\rho$. Its mean is zero, and its multiplier on a harmonic of positive degree $\ell$ is
\[
 \frac1{d-1}\int_0^1(1-\rho^{d-1})\rho^{\ell-1}\,d\rho
 =\frac1{\ell(\ell+d-1)}.
\]
Thus~\eqref{eq:generalGreenPoisson} gives precisely the inverse of $-\Delta_{S^d}$ on functions of mean zero. Differentiation under the integral is justified on every compact subinterval of $(-1,1)$. Each positive-order derivative of $\mathcal P_{\rho,d}$ with respect to $t$ is strictly positive, so
\begin{equation}\label{eq:generalGreenDerivatives}
 G_d^{(k)}(t)>0,\qquad -1<t<1,
\end{equation}
for every positive integer $k$.

\begin{lemma}\label{lem:monic}
Let $w$ be an integrable weight that is positive almost everywhere on
$(-1,1)$. There is a unique real monic polynomial $q_n$ of degree $n$
minimizing $\int_{-1}^1|q(t)|w(t)\,dt$. Its zeros are simple and lie
in $(-1,1)$, and
\begin{equation}\label{eq:orthog}
 \int_{-1}^1p(t)\sgn q_n(t)w(t)\,dt=0,
 \qquad \deg p<n.
\end{equation}
If $w$ is even, then $q_n(-t)=(-1)^nq_n(t)$.
\end{lemma}
\begin{proof}
Existence follows from coercivity on the finite-dimensional affine
space of monic polynomials. Differentiating the norm along $q_n+\varepsilon p$
gives \eqref{eq:orthog}; the zero set of $q_n$ has measure zero, so
dominated convergence justifies the derivative. If $q_n$ had fewer
than $n$ sign changes in $(-1,1)$, the product of its sign-change
factors, with a suitable overall sign, would be a polynomial of degree
less than $n$ whose sign agrees with that of $q_n$. This contradicts
\eqref{eq:orthog}. Hence all $n$ zeros are simple and interior.

For two minimizers, equality in the convexity inequality for the
$L^1$ norm implies agreement of their signs almost everywhere. Their
simple sign-change zeros must therefore coincide, and the monic
normalization makes the polynomials equal. Applying uniqueness to
$(-1)^nq_n(-t)$ proves the last assertion.
\end{proof}

For a positive integer $n$, let $q_{n,d}$ be the monic polynomial of degree $n$ minimizing $\int_{-1}^1|q(t)|\,d\mu_d(t)$. Lemma~\ref{lem:monic} applies to this weight: the minimizer is unique, all its zeros are simple and interior, and
\begin{equation}\label{eq:generalSignMoments}
 \int_{-1}^1p(t)\sgn q_{n,d}(t)\,d\mu_d(t)=0
\end{equation}
for every polynomial $p$ of degree less than $n$. These moment conditions determine the nodes required for the Green interpolation problem.

\begin{theorem}\label{thm:BohrAllDimensions}
Let $d\geq2$ and $n\geq1$ be integers, and let $p_{n-1,d}$ interpolate $G_d$ at the zeros of $q_{n,d}$. Then $p_{n-1,d}$ is a best $L^1(\mu_d)$ approximant to $G_d$ among polynomials of degree at most $n-1$, and
\begin{equation}\label{eq:BohrAllDimensions}
 \begin{aligned}
 H(n,S^d)
 &=\int_{-1}^1|G_d(t)-p_{n-1,d}(t)|\,d\mu_d(t)\\
 &=\int_{-1}^1G_d(t)\sgn q_{n,d}(t)\,d\mu_d(t).
 \end{aligned}
\end{equation}
In particular, this is the sharp constant in
\[
 \norm{f}_\infty\leq H(n,S^d)\norm{\Delta_{S^d}f}_\infty
\]
for smooth functions $f\perp\mathcal H_{<n}$, at each fixed $n$.

For a fixed pole $e$, let $f_{n,d}$ be the mean-zero weak solution of
\begin{equation}\label{eq:generalBohrExtremal}
 -\Delta_{S^d}f_{n,d}(\omega)=\sgn q_{n,d}(e\cdot\omega).
\end{equation}
It has no harmonics of degrees less than $n$, and
\begin{equation}\label{eq:generalBohrExtremalValue}
 f_{n,d}(e)=\norm{f_{n,d}}_\infty=H(n,S^d).
\end{equation}
Its smooth regularizations prove sharpness also in the smooth class defining the constant.
\end{theorem}

\begin{proof}
Write $t_1,\ldots,t_n$ for the zeros of $q_{n,d}$. At a point $t\in(-1,1)$ distinct from the nodes, the interpolation remainder is
\[
 G_d(t)-p_{n-1,d}(t)
 =\frac{G_d^{(n)}(\xi)}{n!}\,q_{n,d}(t),
\]
where $\xi$ lies between the smallest and largest of $t,t_1,\ldots,t_n$. By~\eqref{eq:generalGreenDerivatives}, the error therefore has the sign of $q_{n,d}$. Put $s_{n,d}=\sgn q_{n,d}$. For any competing polynomial $p$ of degree less than $n$, the moment conditions give
\[
 \int_{-1}^1|G_d-p|\,d\mu_d
 \geq\int_{-1}^1(G_d-p)s_{n,d}\,d\mu_d
 =\int_{-1}^1G_ds_{n,d}\,d\mu_d,
\]
first for real $p$. For a complex $p$, replacing it by its real part cannot increase its distance from the real kernel. The interpolant attains this lower bound because its error has sign $s_{n,d}$. This proves best approximation and the equality of the two integrals in~\eqref{eq:BohrAllDimensions}.

Now let $f\perp\mathcal H_{<n}$ and put $g=-\Delta f$. The same low harmonics of $g$ vanish. Subtracting the interpolating polynomial from the Green representation at any pole $e$ gives
\[
 f(e)=\int_{S^d}\bigl(G_d(e\cdot\omega)-p_{n-1,d}(e\cdot\omega)\bigr)
 g(\omega)\,d\sigma_d(\omega).
\]
The first integral in~\eqref{eq:BohrAllDimensions} is therefore an upper bound for the Bohr constant.

For the reverse bound, take $g_0(\omega)=s_{n,d}(e\cdot\omega)$. This function is bounded by one and is orthogonal to $\mathcal H_{<n}$. To see the orthogonality, average a spherical harmonic of degree less than $n$ over the rotations fixing $e$. The average is a polynomial of that degree in $e\cdot\omega$, and~\eqref{eq:generalSignMoments} applies. The mean-zero Green potential of $g_0$ is the function $f_{n,d}$ in~\eqref{eq:generalBohrExtremal}; convolution preserves the missing harmonics. This potential is continuous, since rotations of an integrable zonal kernel are continuous in its integral norm. At $e$ it takes exactly the value of the second integral in~\eqref{eq:BohrAllDimensions}. The same kernel bound applies at every pole to this bounded source, proving~\eqref{eq:generalBohrExtremalValue}.

Finally, apply the spherical heat semigroup to $g_0$. It preserves the missing harmonics and the bound by one, and produces smooth functions $g_\varepsilon$. Their mean-zero Green potentials $f_\varepsilon$ are smooth, have the same missing harmonics, and satisfy $-\Delta f_\varepsilon=g_\varepsilon$. The smoothed sign functions converge to $g_0$ away from its finitely many latitude spheres of discontinuity. Since $G_d\in L^1(\mu_d)$, dominated convergence gives
\[
 f_\varepsilon(e)\longrightarrow
 \int_{-1}^1G_d(t)s_{n,d}(t)\,d\mu_d(t).
\]
This proves the matching lower bound within the smooth class used to define $H(n,S^d)$.
\end{proof}

The extremal can therefore be chosen real, although the estimate holds for complex-valued functions. In every dimension, symmetry gives $q_{1,d}(t)=t$. Another simple example is
\[
 q_{3,d}(t)=t\bigl(t^2-1+2^{-2/d}\bigr).
\]
Indeed, parity leaves only the moment against $t$ to check. If $b$ is the positive zero, that condition is $(1-b^2)^{d/2}=1/2$.

Thus the exact fixed-$n$ problem is solved by interpolation at the zeros of $q_{n,d}$. Before computing the constants on $S^2$ and $S^3$, we record a useful closed form for the Green kernel in every even dimension.

\section{The Green kernel in even dimensions}\label{sec:evenGreenKernels}

In even dimensions the Green kernel is a finite sum of poles and a logarithm. This structure concerns the dimension $d$; the harmonic cutoff $n$ remains arbitrary. It will later make the limiting Bohr constants accessible in dimensions four and six.

\begin{proposition}\label{prop:evenGreenStructure}
For an integer $s\geq1$, the mean-zero Green kernel on $S^{2s}$ is
\begin{equation}\label{eq:evenGreenStructure}
 G_{2s}(t)=\sum_{j=1}^{s-1}\frac{b_j}{(1-t)^j}
 -\frac1{2s-1}\log\frac{1-t}{2}+C_s.
\end{equation}
The constant $C_s$ makes the mean zero. When $s\geq2$, the coefficients of the pole terms are positive and satisfy
\begin{equation}\label{eq:evenGreenRecurrence}
 b_1=\frac1{2s-1},\qquad
 b_j=\frac{2(j-1)(s-j)}{j(2s-j-1)}b_{j-1},\qquad 2\leq j\leq s-1.
\end{equation}
For $s=1$ the sum is empty and $C_1=-1$.
\end{proposition}
\begin{proof}
Away from the pole, the mean-zero Green kernel satisfies
\[
 (1-t^2)G_{2s}''(t)-2stG_{2s}'(t)=1.
\]
Substituting the right-hand side of~\eqref{eq:evenGreenStructure} into this equation gives~\eqref{eq:evenGreenRecurrence}. More explicitly, the zonal Laplacian sends $(1-t)^{-j}$ to
\[
 2j(j+1-s)(1-t)^{-j-1}+j(2s-j-1)(1-t)^{-j},
\]
and sends $-\log((1-t)/2)$ to $(2s-1)-(2s-2)/(1-t)$. Cancellation of the first pole gives $b_1$, and cancellation of each subsequent pole gives the recurrence. The highest possible pole cancels because its coefficient contains $j+1-s$ with $j=s-1$. When $s=1$ the logarithm alone satisfies the equation.

The resulting function is regular at $t=-1$ and integrable against $\mu_{2s}$. Writing the differential equation in divergence form gives
\[
 \frac{d}{dt}\left[c_{2s}(1-t^2)^sG_{2s}'(t)\right]
 =c_{2s}(1-t^2)^{s-1}.
\]
The flux vanishes at $t=-1$ and tends to one at $t=1$, since the measure on the right has total mass one. Thus the distributional equation is $-\Delta_{S^{2s}}G_{2s}=\delta_e-1$. Fixing the mean identifies the function with the Green kernel in~\eqref{eq:generalGreenPoisson}. Positivity of the coefficients follows directly from the recurrence.
\end{proof}

Closed forms for the compact-sphere Green problem, with different expressions in even and odd dimensions, are also given by Chapling~\cite{Chapling}. Here the measure has total mass one and the kernel has mean zero.

For the two higher even dimensions used below, the formulas are
\begin{equation}\label{eq:S4Green}
 d\mu_4(t)=\frac34(1-t^2)\,dt,\qquad
 G_4(t)=\frac13\left(\frac1{1-t}-\log\frac{1-t}{2}\right)-\frac79,
\end{equation}
and
\begin{equation}\label{eq:S6Green}
 \begin{aligned}
 d\mu_6(t)&=\frac{15}{16}(1-t^2)^2\,dt,\\
 G_6(t)&=\frac1{15(1-t)^2}+\frac1{5(1-t)}
 -\frac15\log\frac{1-t}{2}-\frac{43}{75}.
 \end{aligned}
\end{equation}
The nonconstant terms have means $7/9$ and $43/75$, respectively. As a first application of the general Bohr formula, $q_{1,4}(t)=t$ gives
\[
 H(1,S^4)=\frac16+\frac13\log2.
\]
We now turn to the complete calculation on $S^2$ and then to $S^3$.

\section{Best \texorpdfstring{$L^1$}{L1} approximation of the zonal kernel on \texorpdfstring{$S^2$}{S2}}\label{sec:zonalapproximation}

On $S^2$, a zonal kernel depends only on $t=x\cdot y$, where $x,y\in S^2$. Normalized surface measure induces the measure $dt/2$ on $[-1,1]$. We retain this normalization and consider the kernel
\begin{equation}\label{eq:G2}
 G_2(t)=-1-\log\frac{1-t}{2}.
\end{equation}
Fix a pole $e\in S^2$ and put $t=e\cdot\omega=\cos\theta$. Normalized surface measure satisfies
\[
 \int_{S^2}a(e\cdot\omega)\,d\sigma_2(\omega)=\frac12\int_{-1}^1a(t)\,dt.
\]
For a function depending only on $t$, the spherical Laplacian is $\bigl((1-t^2)a'(t)\bigr)'$. Let $P_\ell$ be the Legendre polynomial normalized by $P_\ell(1)=1$. Averaging a spherical harmonic $Y\in\mathcal H_\ell$ over rotations fixing $e$ gives $Y(e)P_\ell(t)$: the average is a zonal harmonic of degree $\ell$ and has value $Y(e)$ at the pole. Hence, for every integrable function $a$ on $[-1,1]$,
\begin{equation}\label{eq:S2zonal}
 \int_{S^2}a(e\cdot\omega)Y(\omega)\,d\sigma_2(\omega)=\frac{Y(e)}2\int_{-1}^1a(t)P_\ell(t)\,dt.
\end{equation}

We first verify the mean and the inverse property of $G_2$. Substitution $u=(1-t)/2$ gives $\frac12\int_{-1}^1G_2(t)\,dt=0$. For $\ell\geq1$, the Legendre equation and integration by parts give
\[
 \ell(\ell+1)\int_{-1}^1G_2(t)P_\ell(t)\,dt
 =\int_{-1}^1(1+t)P_\ell'(t)\,dt=2.
\]
Indeed, $G_2'(t)=(1-t)^{-1}$, the boundary term involving $(1-t^2)G_2(t)P_\ell'(t)$ vanishes at both endpoints, and $\int_{-1}^1P_\ell(t)\,dt=0$. It follows from~\eqref{eq:S2zonal} that convolution with $G_2$ multiplies the harmonic of degree $\ell$ by $1/(\ell(\ell+1))$ and annihilates constants. This verifies the Green representation~\eqref{eq:GreenRepresentation} on $S^2$. For a smooth function the harmonic expansion converges uniformly, so the passage from individual harmonics to the integral representation is justified.

We can now address the approximation problem identified in Section~\ref{sec:greenrepresentation}. We shall find the best $L^1(dt/2)$ approximation of $G_2$ by polynomials of degree at most $n-1$, by first solving a trigonometric problem on an equally spaced grid.

Put $t=\cos\theta$. The change of variable introduces the factor $\sin\theta$, which we include in the function to be approximated:
\begin{equation}\label{eq:trigGreen}
 \mathcal G_2(\theta)=G_2(\cos\theta)\sin\theta
 =-\sin\theta-2\sin\theta\log\left|\sin\frac\theta2\right|.
\end{equation}
The value at an integer multiple of $2\pi$ is defined by continuity. Thus $\mathcal G_2$ is a continuous odd $2\pi$-periodic function. For an algebraic polynomial $p$ of degree at most $n-1$, the function $\sin\theta\,p(\cos\theta)$ is a sine polynomial of degree at most $n$, and
\begin{equation}\label{eq:zonaltrigchange}
 \frac12\int_{-1}^1|G_2(t)-p(t)|\,dt
 =\frac14\int_{-\pi}^{\pi}
 |\mathcal G_2(\theta)-\sin\theta\,p(\cos\theta)|\,d\theta.
\end{equation}
We have therefore arrived at an ordinary, unweighted trigonometric $L^1$ approximation problem. The factor $1/4$ records the original spherical normalization and the passage from a half-period to a full period.

The logarithmic Fourier series gives the particularly simple expansion
\begin{equation}\label{eq:trigGreenFourier}
 \mathcal G_2(\theta)=\left(2\log2-\frac32\right)\sin\theta
 +2\sum_{k=2}^{\infty}\frac{\sin(k\theta)}{k^2-1}.
\end{equation}
Indeed, multiply $-\log(2|\sin(\theta/2)|)=\sum_{k\geq1}\cos(k\theta)/k$ by $2\sin\theta$ and collect the sine terms. The series in~\eqref{eq:trigGreenFourier} converges absolutely and uniformly. The first sine term can be reproduced exactly; the approximation problem is governed by the remaining terms.

\begin{theorem}\label{thm:trigGreenApproximation}
For a positive integer $n$, let $t_n$ be the sine polynomial of degree at most $n$ which interpolates $\mathcal G_2$ at
\begin{equation}\label{eq:uniformGreenNodes}
 \theta_j=\frac{j\pi}{n+1},\qquad j=1,\ldots,n.
\end{equation}
Then $t_n$ is a best $L^1(-\pi,\pi)$ approximant among all trigonometric polynomials of degree at most $n$. Its error is
\begin{equation}\label{eq:trigGreenBestError}
 \inf_{t\in T_n}\int_{-\pi}^{\pi}|\mathcal G_2(\theta)-t(\theta)|\,d\theta
 =8\sum_{j=0}^{\infty}
 \frac{1}{(2j+1)\bigl((n+1)^2(2j+1)^2-1\bigr)}.
\end{equation}
Away from the interpolation nodes, the error has the sign of $\sin((n+1)\theta)$ on $(0,\pi)$.
\end{theorem}

\begin{proof}
The sine interpolant exists and is unique by discrete sine orthogonality on the grid~\eqref{eq:uniformGreenNodes}. Oddness extends the interpolation to the equally spaced grid on the full circle; the values at $0$ and $\pi$ are automatically zero. We first show why the error changes sign precisely at these nodes. The useful observation is that the coefficients in~\eqref{eq:trigGreenFourier} are positive integrals of geometric sequences.

For $0<\rho<1$, put
\[
 F_\rho(\theta)=\frac{\sin\theta}{1-2\rho\cos\theta+\rho^2}
 =\sum_{k=1}^{\infty}\rho^{k-1}\sin(k\theta),
\]
and denote sine interpolation on our grid by $\mathcal I_n$. Write $N=n+1$. On this grid, the frequencies $2mN+k$ and $2mN-k$ give, respectively, $\sin(k\theta_j)$ and $-\sin(k\theta_j)$, while multiples of $N$ vanish. Summing the resulting geometric sequences gives
\[
 \mathcal I_nF_\rho(\theta)
 =\sum_{k=1}^{n}\frac{\rho^{k-1}-\rho^{2N-k-1}}{1-\rho^{2N}}\sin(k\theta).
\]
Subtracting this finite sum from $F_\rho$ yields
\begin{equation}\label{eq:geometricSineRemainder}
 F_\rho(\theta)-\mathcal I_nF_\rho(\theta)
 =\frac{\rho^n(1-\rho^2)\sin(N\theta)}
 {(1-\rho^{2N})(1-2\rho\cos\theta+\rho^2)}.
\end{equation}
Every factor apart from $\sin(N\theta)$ is positive.

For $k\geq2$, we have $2/(k^2-1)=\int_0^1(1-\rho^2)\rho^{k-2}\,d\rho$. Consequently~\eqref{eq:trigGreenFourier} can be written as
\[
 \mathcal G_2(\theta)=\left(2\log2-\frac32\right)\sin\theta
 +\int_0^1\frac{1-\rho^2}{\rho}\bigl(F_\rho(\theta)-\sin\theta\bigr)\,d\rho.
\]
The summable Fourier coefficients justify integration term by term. Interpolation commutes with this integral because it is a finite linear combination of the values at the nodes. Since it reproduces $\sin\theta$, formula~\eqref{eq:geometricSineRemainder} gives
\begin{equation}\label{eq:positiveTrigRemainder}
 \mathcal G_2(\theta)-t_n(\theta)
 =\sin((n+1)\theta)\int_0^1
 \frac{\rho^{n-1}(1-\rho^2)^2\,d\rho}
 {(1-\rho^{2n+2})(1-2\rho\cos\theta+\rho^2)}.
\end{equation}
For $0<\theta<\pi$, the integral is finite and strictly positive. This proves the asserted sign of the error; oddness gives the same sign identity on the negative half-period.

The reason this sign pattern proves best approximation is that the function $s(\theta)=\sgn\sin((n+1)\theta)$ annihilates $T_n$. Indeed, its Fourier series is
\begin{equation}\label{eq:S2squarewave}
 s(\theta)=\frac4\pi\sum_{j=0}^{\infty}
 \frac{\sin((2j+1)(n+1)\theta)}{2j+1},
\end{equation}
and all its frequencies exceed $n$. Therefore, for every $t\in T_n$,
\[
 \int_{-\pi}^{\pi}|\mathcal G_2-t|\,d\theta
 \geq\left|\int_{-\pi}^{\pi}(\mathcal G_2-t)s\,d\theta\right|
 =\int_{-\pi}^{\pi}|\mathcal G_2-t_n|\,d\theta.
\]
This proves minimality. Finally, insert~\eqref{eq:trigGreenFourier} and~\eqref{eq:S2squarewave} into $\int_{-\pi}^{\pi}\mathcal G_2s\,d\theta$. Orthogonality of the sine functions gives~\eqref{eq:trigGreenBestError}; the calculation is justified by $L^2$ convergence.
\end{proof}

We now return to algebraic polynomials. Write $t_n(\theta)=\sum_{k=1}^n b_k\sin(k\theta)$ and use the Chebyshev polynomials of the second kind, defined by $\sin(k\theta)=\sin\theta\,U_{k-1}(\cos\theta)$. Then
\[
 p_{n-1}(t)=\sum_{k=1}^n b_kU_{k-1}(t)
\]
is the polynomial of degree at most $n-1$ which interpolates $G_2$ at $t_j=\cos(j\pi/(n+1))$. Conversely, every algebraic polynomial of this degree corresponds to a sine polynomial in this way. Formula~\eqref{eq:zonaltrigchange} and Theorem~\ref{thm:trigGreenApproximation} therefore show that $p_{n-1}$ is a best approximant and give
\begin{equation}\label{eq:zonalBestApproximation}
 \inf_{\deg p\leq n-1}\frac12\int_{-1}^1|G_2(t)-p(t)|\,dt
 =2\sum_{j=0}^{\infty}
 \frac{1}{(2j+1)\bigl((n+1)^2(2j+1)^2-1\bigr)}.
\end{equation}
Thus the Chebyshev nodes arise from equally spaced angles in the trigonometric problem.

We shall also need the sign information. Put $s_n(t)=\sgn U_n(t)$. Formula~\eqref{eq:positiveTrigRemainder} gives $\sgn(G_2-p_{n-1})=s_n$ almost everywhere. Since $p(\cos\theta)\sin\theta$ is a sine polynomial of degree at most $n$, the same orthogonality argument gives
\begin{equation}\label{eq:S2annihilator}
 \int_{-1}^1p(t)s_n(t)\,dt=0,\qquad \deg p<n.
\end{equation}
In particular,
\begin{equation}\label{eq:S2besterror}
 \frac12\int_{-1}^1|G_2(t)-p_{n-1}(t)|\,dt
 =\frac12\int_{-1}^1G_2(t)s_n(t)\,dt.
\end{equation}
The sign which proves best approximation of the kernel will also provide the extremal function for the Bohr estimate.

\section{The Bohr constant on \texorpdfstring{$S^2$}{S2}}

We return to the constant $H(n,S^2)$ defined in~\eqref{eq:sphericalBohrFavard}. The kernel computation gives the exact bound for each $n$ and then the best coefficient independent of $n$.

\begin{theorem}\label{thm:sphere2}
For every positive integer $n$,
\begin{equation}\label{eq:H22fixed}
 H(n,S^2)=2\sum_{j=0}^{\infty}\frac{1}{(2j+1)\bigl((n+1)^2(2j+1)^2-1\bigr)}.
\end{equation}
Consequently,
\begin{equation}\label{eq:H22}
 H(S^2)=\lim_{n\to\infty}n^2H(n,S^2)=\frac74\zeta(3).
\end{equation}
In particular, every smooth function $f$ on $S^2$ with $f\perp\mathcal H_{<n}$ and $\norm{\Delta f}_\infty\leq1$ satisfies
\begin{equation}\label{eq:S2Bohr}
 \norm{f}_\infty\leq\frac{7\zeta(3)}{4n^2},
\end{equation}
and the coefficient $7\zeta(3)/4$ cannot be decreased uniformly in $n$.
\end{theorem}

\begin{proof}
Put $g=-\Delta f$ and fix a pole $e\in S^2$. If $f\perp\mathcal H_{<n}$ and $\norm{g}_\infty\leq1$, estimate~\eqref{eq:kernelApproximationBound} applies with $t_{n-1}=0$. Taking the interpolating polynomial $p_{n-1}$ from Section~\ref{sec:zonalapproximation} and using~\eqref{eq:zonalBestApproximation} gives the asserted upper bound. The matching lower bound comes from the sign of the kernel's interpolation error.

To prove sharpness in the smooth class, start with the bounded zonal function $g_0(\omega)=s_n(e\cdot\omega)$. Formula~\eqref{eq:S2zonal} and~\eqref{eq:S2annihilator} show that $g_0\perp\mathcal H_{<n}$. We approximate it by smooth functions without increasing its norm or introducing low harmonics. For a positive integer $m$, set
\[
 a_m(t)=(m+1)\left(\frac{1+t}{2}\right)^m,
 \qquad
 g_m(\omega)=\int_{S^2}a_m(\omega\cdot\eta)g_0(\eta)\,d\sigma_2(\eta).
\]
The kernel is nonnegative and $\frac12\int_{-1}^1a_m(t)\,dt=1$, so $\norm{g_m}_\infty\leq1$. Each $g_m$ is a spherical polynomial. Formula~\eqref{eq:S2zonal} shows that this convolution acts by a scalar on every harmonic subspace, hence $g_m\perp\mathcal H_{<n}$. Finally, for $0<\delta<2$, its mass on $t\leq1-\delta$ is $(1-\delta/2)^{m+1}$. Consequently $g_m$ tends to $g_0$ at every continuity point of $g_0$, and therefore almost everywhere.

Let $f_m$ be the mean-zero Green potential of $g_m$. It is a smooth spherical polynomial, satisfies $-\Delta f_m=g_m$, and has no harmonics of degrees below $n$. Dominated convergence in~\eqref{eq:GreenRepresentation} now gives
\[
 f_m(e)\longrightarrow\int_{S^2}G_2(e\cdot\omega)g_0(\omega)\,d\sigma_2(\omega)
 =\frac12\int_{-1}^1G_2(t)s_n(t)\,dt.
\]
Every $f_m$ is admissible in~\eqref{eq:sphericalBohrFavard}. The lower bound obtained from $\norm{f_m}_\infty\geq|f_m(e)|$ matches~\eqref{eq:S2besterror}. We have proved
\begin{equation}\label{eq:H22integral}
 H(n,S^2)=\frac12\int_0^\pi G_2(\cos\theta)\sin\theta\,\sgn\sin((n+1)\theta)\,d\theta.
\end{equation}

The value of this integral was computed in Section~\ref{sec:zonalapproximation}. Combining~\eqref{eq:S2besterror} with~\eqref{eq:zonalBestApproximation} proves~\eqref{eq:H22fixed}.

For each positive odd integer $k$,
\[
 0\leq\frac{n^2}{k\bigl((n+1)^2k^2-1\bigr)}\leq\frac1{k^3},
 \qquad
 \lim_{n\to\infty}\frac{n^2}{k\bigl((n+1)^2k^2-1\bigr)}=\frac1{k^3}.
\]
Dominated convergence in~\eqref{eq:H22fixed} gives
\[
 \lim_{n\to\infty}n^2H(n,S^2)=2\sum_{j=0}^{\infty}\frac1{(2j+1)^3}=\frac74\zeta(3).
\]
The same termwise bound shows that $n^2H(n,S^2)\leq7\zeta(3)/4$ for every $n$. Since the exact fixed-$n$ bounds were obtained within the smooth class, their limit proves both~\eqref{eq:H22} and the sharpness of~\eqref{eq:S2Bohr}.
\end{proof}

For the first three values of $n$, formula~\eqref{eq:H22integral} gives
\[
 \begin{aligned}
 H(1,S^2)&=\log2,\\
 H(2,S^2)&=\frac32\log3-2\log2,\\
 H(3,S^2)&=2\log2-\sqrt2\log(1+\sqrt2).
 \end{aligned}
\]

\section{Best \texorpdfstring{$L^1$}{L1} approximation of the zonal kernel on \texorpdfstring{$S^3$}{S3}}\label{sec:zonalS3}

We specialize Theorem~\ref{thm:BohrAllDimensions} to $S^3$. The weight induced by surface measure differs from that on $S^2$, and the interpolation nodes are again characterized by the corresponding $L^1$ extremal polynomial. A positive integral representation of the Green remainder will give both the exact constants at each $n$ and the starting point for their asymptotic analysis.

For a fixed pole $e$, put $t=e\cdot\omega=\cos\theta$. The induced probability measure is
\begin{equation}\label{eq:measures}
 d\mu_3(t)=\frac2\pi\sqrt{1-t^2}\,dt.
\end{equation}
In the angular variable, this measure is $(2/\pi)\sin^2\theta\,d\theta$ on $[0,\pi]$. The mean-zero Green kernel of $-\Delta_{S^3}$ is
\begin{equation}\label{eq:G3}
 G_3(\cos\theta)=\frac12(\pi-\theta)\cot\theta-\frac14.
\end{equation}
For $0<\theta<\pi$, direct differentiation gives $\Delta G_3=1$, where the zonal Laplacian is $\partial_\theta^2+2\cot\theta\,\partial_\theta$. Near zero, $G_3(\cos\theta)=\pi/(2\theta)+O(1)$ and its derivative with respect to $\theta$ is $-\pi/(2\theta^2)+O(1)$. The area of the geodesic sphere of radius $\theta$, divided by the area $2\pi^2$ of $S^3$, is $(2/\pi)\sin^2\theta$, so the flux of $-\nabla G_3$ tends to one. Finally, integration against $(2/\pi)\sin^2\theta\,d\theta$ gives mean zero. Thus $-\Delta_{S^3}G_3=\delta_e-1$ with respect to normalized surface measure, proving~\eqref{eq:GreenRepresentation} in this case as well.

Let $f\perp\mathcal H_{<n}$ and put $g=-\Delta f$. Since the constant harmonic is absent,
\[
 f(e)=\int_{S^3}G_3(e\cdot\omega)g(\omega)\,d\sigma_3(\omega).
\]
If $p$ is a polynomial of degree less than $n$, the spectral gap gives
\begin{equation}\label{eq:upperapprox}
 |f(e)|\leq\norm{g}_\infty
 \int_{-1}^1|G_3(t)-p(t)|\,d\mu_3(t).
\end{equation}
To make this bound sharp, we need the sign of the approximation error to annihilate every polynomial of degree less than $n$. We therefore choose the interpolation nodes as the zeros of a monic polynomial whose sign has exactly these vanishing moments. Lemma~\ref{lem:monic} supplies this polynomial.

Apply Lemma~\ref{lem:monic} with $w(t)=\sqrt{1-t^2}$, and let $q_n$ be its monic minimizer. Thus
\begin{equation}\label{eq:q3orthog}
 \int_{-1}^1p(t)\sgn q_n(t)\sqrt{1-t^2}\,dt=0,
 \qquad \deg p<n.
\end{equation}
It is the sign of $q_n$ that has the required vanishing moments. This is the property that will make the interpolation error an extremal kernel.

\begin{theorem}\label{thm:S3fixedn}
For a positive integer $n$, let $p_{n-1}$ interpolate $G_3$ at the zeros of $q_n$. Then $p_{n-1}$ is a best $L^1(\mu_3)$ approximant to $G_3$ among polynomials of degree at most $n-1$, and
\begin{equation}\label{eq:sharp-sign}
 \begin{aligned}
 H(n,S^3)&=\int_{-1}^1|G_3(t)-p_{n-1}(t)|\,d\mu_3(t)\\
 &=\int_{-1}^1G_3(t)\sgn q_n(t)\,d\mu_3(t).
 \end{aligned}
\end{equation}
\end{theorem}

\begin{proof}
The following representation makes the sign of the Green interpolation
error visible:
\begin{equation}\label{eq:Stieltjes}
 G_3(t)=-\frac14+\int_1^\infty\frac{a}{2\sqrt{a^2-1}}
 \left(\frac1{a-t}-\frac1a\right)\,da.
\end{equation}
Indeed, putting $a=\cosh u$ reduces the integral to
$\tfrac t2\int_0^\infty(\cosh u-t)^{-1}\,du$. For $t=\cos\theta$
this equals $\tfrac12(\pi-\theta)\cot\theta$.

Let $p_{n-1}$ interpolate $G_3$ at the zeros of $q_n$. Interpolation
of $(a-t)^{-1}$ has the exact remainder
$q_n(t)/((a-t)q_n(a))$. Therefore
\begin{equation}\label{eq:remainder3}
 G_3(t)-p_{n-1}(t)
 =q_n(t)\int_1^\infty
 \frac{a\,da}{2\sqrt{a^2-1}(a-t)q_n(a)}.
\end{equation}
The integral is strictly positive for $-1<t<1$, since $q_n$ has all its zeros in that interval and is monic. Consequently,
\begin{equation}\label{eq:error-sign}
 \sgn(G_3-p_{n-1})=\sgn q_n.
\end{equation}
Integrating the error of any competing polynomial against $\sgn q_n$ and using~\eqref{eq:q3orthog} gives a lower bound for its $L^1(\mu_3)$ norm. By~\eqref{eq:error-sign}, the interpolant attains that bound. This proves minimality and the second equality in~\eqref{eq:sharp-sign}; estimate~\eqref{eq:upperapprox} gives the corresponding upper bound for $H(n,S^3)$.

It remains to prove sharpness within the smooth class of~\eqref{eq:sphericalBohrFavard}. The zonal function $g(\omega)=\sgn q_n(e\cdot\omega)$
annihilates $\mathcal H_{<n}$: average a spherical harmonic over
rotations fixing $e$ and use \eqref{eq:orthog}. Its Green potential
attains the value on the right of \eqref{eq:sharp-sign} at $e$.
Apply the spherical heat semigroup to $g$. This preserves the gap and
does not increase its supremum norm. The corresponding smooth Green
potentials converge at $e$, since $G_3$ is integrable and the smoothed
sign functions converge almost everywhere. Together with
\eqref{eq:upperapprox}, this proves \eqref{eq:sharp-sign}.

\end{proof}

Formula~\eqref{eq:sharp-sign} determines the exact constant for each $n$, although the polynomial $q_n$ is in general specified by its extremal property. For the first few degrees the moment conditions give
\[
 q_1(t)=t,\qquad q_2(t)=t^2-a^2,\qquad
 q_3(t)=t^3-\bigl(1-2^{-2/3}\bigr)t,
\]
where $a$ is the unique number in $(0,1)$ satisfying
\[
 \arcsin a+a\sqrt{1-a^2}=\frac\pi4.
\]
Indeed, parity leaves only the constant moment to check for $q_2$, and this moment gives the displayed equation. For $q_3(t)=t(t^2-b^2)$, only the moment against $t$ remains; its vanishing gives $(1-b^2)^{3/2}=1/2$. The constants are
\begin{equation}\label{eq:S3smallconstants}
 \begin{aligned}
 H(1,S^3)&=\frac12,\\
 H(2,S^3)&=\frac{2(1-a^2)\arcsin a}{\pi}\approx0.221537610357,\\
 H(3,S^3)&=2^{-2/3}-\frac12\approx0.129960524947.
 \end{aligned}
\end{equation}
For odd $n$, parity and the odd part of $G_3$ reduce~\eqref{eq:sharp-sign} to $H(n,S^3)=\int_0^1t\sgn q_n(t)\,dt$, which gives the first and third values. For $n=2$, the constant term in the even part of $G_3$ cancels by the zero moment; integration of $t\arcsin t$ on the two sign intervals gives the middle value.

The same result can be read as a trigonometric approximation theorem. Put $t=\cos\theta$ and incorporate the factor $\sin^2\theta$ from the weight and the change of variable into the kernel:
\begin{equation}\label{eq:trigGreen3}
 \mathcal G_3(\theta)=G_3(\cos\theta)\sin^2\theta
 =\frac14(\pi-\theta)\sin2\theta-\frac14\sin^2\theta,
 \qquad 0<\theta<\pi.
\end{equation}
Both endpoint values are zero. Extending evenly to $[-\pi,\pi]$ and then periodically gives a continuous function.

An algebraic polynomial $p$ of degree at most $n-1$ corresponds to $t_{n+1}(\theta)=\sin^2\theta\,p(\cos\theta)$. These are exactly the even trigonometric polynomials of degree at most $n+1$ which vanish at $0$ and $\pi$; denote their space by $\mathcal V_{n+1}$. Indeed, every even trigonometric polynomial is a polynomial in $\cos\theta$, and the two endpoint conditions mean divisibility by $1-t^2$. The change of variable and evenness give
\begin{equation}\label{eq:S3trigchange}
 \int_{-1}^1|G_3(t)-p(t)|\,d\mu_3(t)
 =\frac1\pi\int_{-\pi}^{\pi}|\mathcal G_3(\theta)-t_{n+1}(\theta)|\,d\theta.
\end{equation}
Consequently,
\begin{equation}\label{eq:S3trigminimum}
 H(n,S^3)=\frac1\pi\inf_{v\in\mathcal V_{n+1}}
 \int_{-\pi}^{\pi}|\mathcal G_3(\theta)-v(\theta)|\,d\theta.
\end{equation}

If $t_1,\ldots,t_n$ are the zeros of $q_n$, the minimizing polynomial is $\sin^2\theta\,p_{n-1}(\cos\theta)$. It interpolates $\mathcal G_3$ at $\theta_j=\arccos t_j$, as well as at the two endpoints. These conditions determine it uniquely in $\mathcal V_{n+1}$. Its minimality follows directly from Theorem~\ref{thm:S3fixedn} and~\eqref{eq:S3trigchange}. The angular nodes are determined by the weighted moment conditions and need not be equally spaced. The two endpoint conditions are essential: approximation by all trigonometric polynomials of degree at most $n+1$ would be a different problem.

For even degrees, a further change of angle gives a shorter formula for the constant. Put
\begin{equation}\label{eq:S3cosineExtremal}
 P_m(t)=\frac{q_{2m}(\cos(t/2))}{q_{2m}(1)},\qquad
 \ell_m=\int_0^\pi|P_m(t)|\,dt.
\end{equation}
Since $q_{2m}$ is even, $P_m$ is a cosine polynomial of degree $m$, with $P_m(0)=1$. It has the least integral norm among real cosine polynomials of degree at most $m$ with this normalization. Indeed, a variation vanishing at zero is divisible by $1-\cos t$; after the substitution $t=2\theta$, the first variation vanishes by~\eqref{eq:q3orthog}. Convexity gives minimality. Equivalently, with $s_m=\sgn P_m$,
\begin{equation}\label{eq:S3cosineReproduction}
 \int_0^\pi T(t)s_m(t)\,dt=\ell_m T(0)
\end{equation}
for every real cosine polynomial $T$ of degree at most $m$.

\begin{proposition}\label{prop:S3reciprocal}
For every positive integer $m$,
\begin{equation}\label{eq:S3finiteReciprocal}
 H(2m,S^3)=\frac{\ell_m}{4\pi}
 \int_0^\infty\frac{d\eta}{P_m(i\eta)}.
\end{equation}
\end{proposition}

The integrand is positive: every zero of the algebraic polynomial defining $P_m$ lies between $-1$ and $1$, whereas $\cos(i\eta)=\cosh\eta\geq1$. The formula thus replaces integration of a sign-changing kernel by one positive integral.

\begin{proof}
Write $P_m(t)=A_m(\cos t)$, and denote its zeros on $(0,\pi)$ by $t_1<\cdots<t_m$. Put $x_j=\cos t_j$ and $\varepsilon_j=(-1)^{j-1}$. Applying~\eqref{eq:S3cosineReproduction} to $\cos kt$ and integrating by parts gives
\[
 2\sum_{j=1}^m\varepsilon_j\sin(kt_j)=k\ell_m,
 \qquad k=1,\ldots,m.
\]
Since $\sin kt=\sin t\,U_{k-1}(\cos t)$ and $U_{k-1}(1)=k$, these identities reproduce evaluation at $1$ for every algebraic polynomial of degree less than $m$. Applying them to the Lagrange basis at $x_1,\ldots,x_m$, and using $A_m(1)=1$, yields
\[
 \frac1{A_m'(x_j)}
 =\frac{2\varepsilon_j(1-x_j)\sin t_j}{\ell_m}.
\]
Now expand $1/A_m(v)$ in simple fractions and use
\[
 \int_0^\infty\frac{d\eta}{\cosh\eta-\cos t}
 =\frac{\pi-t}{\sin t},\qquad 0<t<\pi.
\]
It follows that the right-hand side of~\eqref{eq:S3finiteReciprocal} is
\[
 \frac1{2\pi}\sum_{j=1}^m
 \varepsilon_j(\pi-t_j)(1-\cos t_j).
\]
To identify this sum with the Bohr constant, pair $\theta$ and $\pi-\theta$ in~\eqref{eq:sharp-sign}, and then put $t=2\theta$. The term involving $1-\cos t$ vanishes by~\eqref{eq:S3cosineReproduction}, leaving
\[
 H(2m,S^3)=\frac1{4\pi}\int_0^\pi
 (\pi-t)\sin t\,s_m(t)\,dt.
\]
Integration by parts over the sign intervals gives the preceding sum: the remaining terms cancel because~\eqref{eq:S3cosineReproduction} gives $\int s_m=\int s_m\cos t=\ell_m$. This proves the formula.
\end{proof}

\section{The asymptotic Bohr constant on \texorpdfstring{$S^3$}{S3}}\label{sec:rescaling}

Having determined the constant for each $n$, we now study $n^2H(n,S^3)$ as $n$ tends to infinity. The relevant angular scale near a pole is $1/n$. We normalize the interpolating-node polynomial to have value one at the pole and rescale its argument:
\begin{equation}\label{eq:FN}
 Q_n(\theta)=\frac{q_n(\cos\theta)}{q_n(1)},
 \qquad F_n(z)=Q_n(z/n).
\end{equation}
The trigonometric degree grows with $n$, but the rescaled functions have exponential type at most one. The reason a specific entire extremal function will appear is that $Q_n$ solves a minimum-norm problem with its value at zero prescribed. To make this precise, put
\[
 \mathcal T_n'=\operatorname{span}\{e^{ik\theta}:k=-n,-n+2,\ldots,n-2,n\}.
\]
Here the frequencies are taken in steps of two and have the same parity as $n$, in accordance with the symmetry $q_n(-t)=(-1)^nq_n(t)$.

\begin{lemma}\label{lem:point}
The function $Q_n$ minimizes
\[
 J(T)=\int_{-\pi/2}^{\pi/2}|T(\theta)|\,d\theta
\]
over $T\in\mathcal T_n'$ with $T(0)=1$. Furthermore,
$\norm{Q_n}_\infty=1$ on $\R$ and $|F_n(z)|\leq e^{|\operatorname{Im}z|}$.
\end{lemma}
\begin{proof}
First take real even $T$. It is a polynomial in $t=\cos\theta$ of
parity $n$. A variation vanishing at $\theta=0$ also vanishes at
$t=-1$, and hence has the form $(1-t^2)p(t)$ with $\deg p\leq n-2$.
After the change of variable $t=\cos\theta$, the first variation of
$J$ at $Q_n$ is zero by \eqref{eq:q3orthog}. Convexity gives
minimality. Taking the real part and then averaging under
$\theta\mapsto-\theta$ reduces an arbitrary complex competitor to
this case without increasing its norm or changing its value at zero.

The modulus of every member of $\mathcal T_n'$ is $\pi$-periodic.
If $|Q_n(\theta_0)|>1$, translating by $\theta_0$ and dividing by
$Q_n(\theta_0)$ produces a competitor with smaller $J$. Thus
$\norm{Q_n}_\infty=1$. The Bernstein bound for trigonometric
polynomials gives $|Q_n(z)|\leq e^{n|\operatorname{Im}z|}$ and
proves the assertion about $F_n$.
\end{proof}

After the substitution $x=n\theta$, the interval in this extremal problem expands to the real line. Its limiting form is the following problem in the space $\PW^1_1$ of entire functions of exponential type at most one which are integrable on $\R$:
\begin{equation}\label{eq:HB}
 M=\min\left\{\norm{F}_{L^1(\R)}:\;F\in\PW^1_1,\quad F(0)=1\right\}.
\end{equation}
The unique minimizer is the H\"ormander--Bernhardsson function $\varphi$, here normalized to type one. We use its existence, uniqueness, and real-zero product as recorded in~\cite{BORSp}:
\begin{equation}\label{eq:product}
 \varphi(z)=\prod_{j=1}^\infty\left(1-\frac{z^2}{a_j^2}\right),
 \qquad 0<a_1<a_2<\cdots.
\end{equation}
Thus $\varphi$ is real and even, and
\begin{equation}\label{eq:imaglower}
 \varphi(iy)\geq1+c_0y^2,\qquad c_0=\sum_{j=1}^\infty a_j^{-2}>0.
\end{equation}
Translation and normalization in the extremal problem give $\norm{\varphi}_\infty=1$. The Bernstein bound then gives $|\varphi(z)|\leq e^{|\operatorname{Im}z|}$.

The constant in the asymptotic formula has the following positive integral representation:
\begin{equation}\label{eq:beta}
 \beta=\frac{M}{2\pi}\int_0^\infty\frac{dy}{\varphi(iy)}.
\end{equation}
Its finiteness follows immediately from~\eqref{eq:imaglower}.

\begin{theorem}\label{thm:S3asymptotic}
The exact constants in Theorem~\ref{thm:S3fixedn} satisfy
\begin{equation}\label{eq:limit3}
 \lim_{n\to\infty}n^2H(n,S^3)=\beta.
\end{equation}
\end{theorem}

Numerical evaluation gives $\beta\approx2.850999422031$. The calculation is described in Section~\ref{app:beta}; this approximation is not used in the proof.

We first prove convergence of the rescaled extremal functions, including control of their $L^1$ mass outside bounded intervals. Local convergence alone would not justify passage to the limit in the integral for the constant.

\begin{proposition}\label{prop:convergence}
The functions $F_n$ converge to $\varphi$ locally uniformly on $\C$.
In addition,
\begin{equation}\label{eq:mass}
 m_n:=\int_{-\pi n/2}^{\pi n/2}|F_n(x)|\,dx\longrightarrow M,
\end{equation}
and their mass on these expanding intervals is tight:
\begin{equation}\label{eq:tight}
 \lim_{R\to\infty}\limsup_{n\to\infty}
 \int_{R<|x|\leq\pi n/2}|F_n(x)|\,dx=0.
\end{equation}
\end{proposition}
\begin{proof}
We first construct competitors. Let $F$ be a real even Schwartz
function of exponential type strictly less than one, with $F(0)=1$.
Its periodization
\[
 T_n(\theta)=\sum_{j\in\mathbb Z}(-1)^{nj}
 F\bigl(n(\theta+j\pi)\bigr)
\]
belongs to $\mathcal T_n'$. Indeed, Poisson summation restricts its
frequencies to integers of parity $n$, and the support of $\widehat F$
restricts their absolute values to less than $n$. Also,
\[
 T_n(0)\longrightarrow1,\qquad
 J(T_n)\leq\frac1n\norm{F}_1.
\]
Applying Lemma~\ref{lem:point} to $T_n/T_n(0)$ for all sufficiently large $n$ gives $\limsup m_n\leq\norm{F}_1$.

Such competitors can approximate $\varphi$ in $L^1$. For example, choose a real even Schwartz function $\psi$ of type at most one with $\psi(0)=1$ and $|\psi|\leq1$, and use $\varphi((1-2\varepsilon)x)\psi(\varepsilon x)$ for $0<\varepsilon<1/2$. Dilation continuity in $L^1$, followed by dominated convergence for the multiplier, proves convergence to $\varphi$. Hence $\limsup m_n\leq M$.

The bound in Lemma~\ref{lem:point} makes $(F_n)$ a normal family. Choose a subsequence along which $m_n$ tends to its lower limit, and then a locally uniformly convergent further subsequence. Its limit $F_*$ has type at most one and $F_*(0)=1$. Fatou's lemma on expanding real intervals gives $\norm{F_*}_1\leq\liminf m_n$, whereas~\eqref{eq:HB} gives $M\leq\norm{F_*}_1$. Together with the upper bound already proved, this shows that $m_n\to M$. For any locally uniform subsequential limit the same argument now gives norm $M$, so uniqueness forces that limit to be $\varphi$. This proves convergence of the full sequence. Subtracting the integral on $[-R,R]$ from~\eqref{eq:mass}, and then letting $R$ grow, proves~\eqref{eq:tight}.
\end{proof}

\begin{proof}[Proof of Theorem~\ref{thm:S3asymptotic}]
It is enough first to consider $n=2m$. The definitions give $P_m(t)=Q_{2m}(t/2)$, $\ell_m=J(Q_{2m})$, and $F_{2m}(z)=P_m(z/m)$. Rescaling~\eqref{eq:S3finiteReciprocal} therefore gives
\[
 (2m)^2H(2m,S^3)=\frac{m_{2m}}{2\pi}
 \int_0^\infty\frac{dy}{F_{2m}(iy)}.
\]
Proposition~\ref{prop:convergence} gives $m_{2m}\to M$ and local convergence of the denominator to $\varphi(iy)$. To pass to the limit in the integral, note that the opposite pairs of roots of $q_{2m}$ give
\[
 F_{2m}(iy)=\prod_{j=1}^m
 \left(1+\frac{\sinh^2(y/(2m))}{1-r_j^2}\right),
\]
where $r_j$ are its positive roots. All coefficients in the power series in $y^2$ are nonnegative. Since $-F_{2m}''(0)/2\to c_0>0$, there is a fixed $c>0$ such that $F_{2m}(iy)\geq1+cy^2$ for all sufficiently large $m$ and all $y\geq0$. Dominated convergence proves the limit along the even degrees.

Finally, $H(n,S^3)$ is nonincreasing in $n$, by its best-kernel-approximation characterization. Thus $H(2m+2,S^3)\leq H(2m+1,S^3)\leq H(2m,S^3)$, and multiplying by $(2m+1)^2$ gives the same limit along the odd degrees.
\end{proof}

For the uniform bound below, we shall also need the equivalent form
\begin{equation}\label{eq:betaDouble}
 \beta=\frac2\pi\int_0^\infty\int_0^\infty
 \frac{x^2|\varphi(x)|}{(x^2+y^2)\varphi(iy)}\,dy\,dx.
\end{equation}
We verify it directly from the extremal property, without another limiting argument. Put $s=\sgn\varphi$. First variation in~\eqref{eq:HB} gives
\begin{equation}\label{eq:reproduce}
 \int_\R F(x)s(x)\,dx=M F(0),\qquad F\in\PW^1_1.
\end{equation}
Indeed, vary $\varphi$ by $F-F(0)\varphi$, first for real functions and then by complex linearity. For $y>0$, write
\[
 A(y)=\int_\R\frac{|\varphi(x)|}{x^2+y^2}\,dx,
 \qquad J_s(y)=\int_\R\frac{s(x)}{x^2+y^2}\,dx,
 \qquad V(y)=\varphi(iy).
\]
Apply~\eqref{eq:reproduce} to
\[
 \frac{\varphi(z)-V(y)}{z^2+y^2}
 \quad\hbox{and}\quad
 \frac{z\varphi'(z)-yV'(y)}{z^2+y^2}.
\]
The poles are removable, and both functions belong to $\PW^1_1$; integrability of the second follows also from Bernstein's inequality for $\varphi'$. Integration by parts in the second identity gives
\[
 A-VJ_s=\frac{M(1-V)}{y^2},\qquad
 A+yA'-yV'J_s=-\frac{MV'}y.
\]
Consequently, if $D(y)=(M-y^2J_s(y))/\pi$, the first identity and its derivative give
\[
 D(y)=\frac1{\pi V(y)}\int_\R
 \frac{x^2|\varphi(x)|}{x^2+y^2}\,dx,
 \qquad yD'(y)-D(y)=-\frac{M}{\pi V(y)}.
\]
In particular, $0\leq D(y)\leq M/(\pi V(y))$. By~\eqref{eq:imaglower}, $yD(y)$ tends to zero at both endpoints. Integrating the differential identity now yields
\begin{equation}\label{eq:beta1D}
 \int_0^\infty D(y)\,dy
 =\frac{M}{2\pi}\int_0^\infty\frac{dy}{V(y)}=\beta.
\end{equation}
The positive expression for $D$ and Tonelli's theorem prove~\eqref{eq:betaDouble}.

\section{A uniform Bohr bound on \texorpdfstring{$S^3$}{S3}}\label{sec:S3uniform}

The limit shows that any coefficient valid for all $n$ must be at least $\beta$, but it does not by itself show that $\beta$ works at every finite $n$. We now prove that no larger coefficient is needed.

\begin{theorem}\label{thm:main}
For every positive integer $n$,
\begin{equation}\label{eq:H23}
 n^2H(n,S^3)\leq\beta.
\end{equation}
Consequently $H(S^3)=\beta$.
\end{theorem}

We shall construct an integrable kernel on $\R^3$ whose Fourier multiplier agrees with the inverse Laplacian for frequencies of magnitude at least one, and transfer the resulting bound to $S^3$. The three-dimensional Euclidean space is useful because its radial waves $\sin(\lambda r)/r$ can be related exactly to the zonal spherical harmonics. To obtain an integrable kernel, we subtract an entire approximant from the Newton kernel. Define
\begin{equation}\label{eq:B}
 B(z)=\frac2\pi\int_0^\infty
 \frac{1-\varphi(z)/\varphi(iy)}{z^2+y^2}\,dy.
\end{equation}
This construction interpolates $1/|x|$ at the real zeros of
$\varphi$ and gives an error with a prescribed sign.

\begin{lemma}\label{lem:B}
The function $B$ is even, entire, and of exponential type at most
one. For $x>0$,
\begin{equation}\label{eq:Berror}
 \frac1x-B(x)=\frac2\pi\varphi(x)
 \int_0^\infty\frac{dy}{(x^2+y^2)\varphi(iy)}.
\end{equation}
The radial function
\begin{equation}\label{eq:K}
 \mathcal K(u)=\frac1{4\pi}\left(\frac1{|u|}-B(|u|)\right),
 \qquad u\in\R^3,
\end{equation}
belongs to $L^1(\R^3)$ and satisfies
\begin{equation}\label{eq:Kfacts}
 \norm{\mathcal K}_1=\beta,\qquad
 \widehat{\mathcal K}(\xi)=|\xi|^{-2}\quad\text{for }|\xi|\geq1.
\end{equation}
Here $\widehat h(\xi)=\int h(u)e^{-iu\cdot\xi}\,du$.
\end{lemma}
\begin{proof}
For fixed $y$, the poles at $z=\pm iy$ in \eqref{eq:B} are removable
because $\varphi$ is even. This remains locally uniform at $z=y=0$:
write $\varphi(z)=H(z^2)$ and use the divided difference of $H$.
For large $y$, the integrand is locally bounded by a constant times
$y^{-2}$. Thus the integral defines an entire even function.

Here is also a growth bound, to make the spectral support explicit.
For $|z|=R\geq2$, split the integral at $2R$. The tail is bounded
by $C e^{|\operatorname{Im}z|}/R$. In $0\leq y\leq2R$, away from
unit neighborhoods of $z=iy$ and $z=-iy$, the two denominator
factors have modulus at least one, giving a bound
$CR e^{|\operatorname{Im}z|}$. In either neighborhood cancel the
nearby factor by the divided-difference formula. Cauchy's estimate
and $|\varphi(z)|\leq e^{|\operatorname{Im}z|}$ bound the resulting
derivative by $Ce^{|\operatorname{Im}z|}$, while the other factor
has modulus comparable to $R$. These bounds prove
\[
 |B(z)|\leq C(1+|z|)e^{|\operatorname{Im}z|}.
\]
The constants here are independent of $z$.

Since $(2/\pi)\int_0^\infty(x^2+y^2)^{-1}\,dy=1/x$,
formula \eqref{eq:Berror} follows. Polar coordinates and~\eqref{eq:betaDouble} then give
\[
 \norm{\mathcal K}_1=\int_0^\infty x^2\left|x^{-1}-B(x)\right|\,dx
 =\beta.
\]
The one-dimensional Paley--Wiener theorem for distributions implies
$\supp\widehat B\subset[-1,1]$. The radial Fourier identity
\[
 \widehat{B(|\cdot|)}_{\R^3}(\rho)
 =-\frac{2\pi}{\rho}\frac d{d\rho}\widehat B_{\R}(\rho),
 \qquad\rho>0,
\]
follows by differentiating the one-dimensional cosine transform in the radial Fourier formula, and remains valid for tempered distributions. It shows that the radial lift has spectrum in the unit ball. The Fourier transform of
$(4\pi|u|)^{-1}$ is $|\xi|^{-2}$. This proves the multiplier
identity outside that ball. Since $\mathcal K\in L^1$, its Fourier transform
is continuous; the same identity holds on the boundary.
\end{proof}

\begin{proof}[Proof of Theorem~\ref{thm:main}]
Let $f\perp\mathcal H_{<n}$. Fix a pole $e$ and average $f$ over rotations fixing $e$. This
preserves $f(e)$, preserves the gap, and does not increase
$\norm{\Delta f}_\infty$. Thus we may assume that $f$ and
$g=-\Delta f$ are zonal. Write $g(x)=\gamma(e\cdot x)$. The zonal harmonics on $S^3$ are
$U_\ell(\cos\theta)=\sin((\ell+1)\theta)/\sin\theta$. We may write
\begin{equation}\label{eq:zonalexp}
 \gamma(\cos\theta)\sin\theta=\sum_{k\geq n+1}b_k\sin k\theta,
 \qquad
 f(e)=\sum_{k\geq n+1}\frac{k b_k}{k^2-1}.
\end{equation}
For smooth $g$ the coefficients decrease rapidly. Multiplication by $\cos r$ will replace each frequency $k$ by $k-1$ and $k+1$. Dividing by $r$ then gives the radial Euclidean waves on which the inverse multiplier is known. This motivates the bounded function on $\R^3$
\begin{equation}\label{eq:h}
 h(u)=\frac{\sin r\cos r}{r}\gamma(\cos r),\qquad r=|u|,
\end{equation}
with the continuous value at zero. Then
$\norm{h}_\infty\leq\norm{g}_\infty$ and
\begin{equation}\label{eq:hshells}
 h(u)=\frac12\sum_{k\geq n+1}b_k
 \left(\frac{\sin((k+1)r)}r+\frac{\sin((k-1)r)}r\right).
\end{equation}
All frequencies here are at least $n$.

Scale the kernel by $\mathcal K_n(u)=n\mathcal K(nu)$. Then
$\norm{\mathcal K_n}_1=\beta/n^2$ and
$\widehat{\mathcal K_n}(\xi)=|\xi|^{-2}$ for $|\xi|\geq n$.
The spherical average of a plane wave is
\[
 \int_{S^2}e^{i\lambda\omega\cdot u}\,d\sigma_2(\omega)
 =\frac{\sin(\lambda r)}{\lambda r}.
\]
Consequently, for $\lambda\geq n$,
\[
 \int_{\R^3}\mathcal K_n(u)\frac{\sin(\lambda|u|)}{|u|}\,du
 =\lambda\widehat{\mathcal K_n}(\lambda)=\frac1\lambda.
\]
Integrating \eqref{eq:hshells} term by term is justified by rapid
decay of the coefficients and $\mathcal K_n\in L^1$. Using
\eqref{eq:zonalexp}, we obtain
\[
 \int_{\R^3}\mathcal K_n(u)h(u)\,du
 =\frac12\sum_{k\geq n+1}b_k
 \left(\frac1{k+1}+\frac1{k-1}\right)=f(e).
\]
It follows that
\[
 |f(e)|\leq\frac\beta{n^2}\norm{g}_\infty.
\]
The pole was arbitrary, so this proves~\eqref{eq:H23}. The limit in Theorem~\ref{thm:S3asymptotic} shows that the coefficient cannot be decreased, and hence $H(S^3)=\beta$.

\end{proof}

\section{Numerical evaluation of the limiting constant}\label{app:beta}

The positive integral~\eqref{eq:beta} gives a direct way to evaluate $\beta$. For reproducibility, we also describe the equivalent phase formula used for the numerical value quoted after Theorem~\ref{thm:S3asymptotic}. By~\eqref{eq:beta1D}, we may integrate $D(y)=(M-y^2J_s(y))/\pi$. The numerical calculation is not used in any of the proofs.

For the type-$\pi$ normalization set $\tau_j=a_j/\pi$ and
$\Phi(z)=\prod_{j\geq1}(1+(-1)^jz/\tau_j)$. Then
$\varphi(z)=\Phi(z/\pi)\Phi(-z/\pi)$. Write $c=\pi/M$ and
$\ell=\Phi'(0)$. The differential equation and numerical parameters
in~\cite[Theorem 1.1 and Section 9]{BORS} are
\begin{gather*}
 z^2\Phi''+(2z-(2c)^{-1})\Phi'
 +(\pi^2z^2/4+\ell/(2c))\Phi=0,\qquad\Phi(0)=1,\\
 c\approx0.5409288219018305893928820589996903868550,\\
 \ell\approx-0.4519521648844099974932868451365782916109.
\end{gather*}
Integration over successive sign intervals yields
\[
 J_s(y)=\frac{\pi+4\arg\Phi(iy/\pi)}y,
\]
where the argument is continued from zero. Thus the integrand in
\eqref{eq:beta1D} is
\[
 D(y)=\frac{M-\pi y-4y\arg\Phi(iy/\pi)}\pi.
\]
For large $y$ this expression suffers cancellation. The functional
equation in~\cite[Theorem 1.1]{BORS} instead gives
\[
 D(y)=\frac{4y}\pi\arctan\left(e^{-y}
 \frac{\Phi(-1/(2cy))}{\Phi(1/(2cy))}\right)
\]
as soon as $1/(2cy)<\tau_1$. This formula applies for $y\geq1$
with the parameters above.

For reproducibility, if $\Phi(z)=\sum u_nz^n$, set $u_0=1$ and
$u_{-1}=u_{-2}=0$. The differential equation gives
\[
 u_{n+1}=\frac{2c}{n+1}
 \left(\left(n(n+1)+\frac\ell{2c}\right)u_n
 +\frac{\pi^2}{4}u_{n-2}\right).
\]
The 100-digit values of $c$ and $\ell$ in~\cite[Section 9]{BORS},
rather than the rounded values above, were used with 120-digit
arithmetic to form the coefficients. Forty Taylor coefficients, with the functional equation
for $y\geq1$, and Gauss--Legendre quadrature on
$[0,1]$, $[1,4]$, $[4,12]$, $[12,32]$, and $[32,64]$ give
$2.85099942203104$ with 24, 48, and 96 nodes per interval in the
calculation used here. This agreement is a numerical check, not a
certified error bound for the displayed digits.

\section{Limiting extremal polynomials in higher dimensions}\label{sec:weightedEntireLimit}

The H\"ormander--Bernhardsson function appeared on $S^3$ because the interpolation nodes, after rescaling near the pole, were governed by an unweighted extremal problem on the line. We now consider dimensions greater than three, where the same passage introduces a weight depending on $d$. We first identify that weight and prove convergence of the extremal polynomials. Throughout this section the dimension $d\geq3$ is fixed; constants in the estimates may depend on $d$.

Normalize the polynomial from Theorem~\ref{thm:BohrAllDimensions} at the pole and put
\begin{equation}\label{eq:generalEndpointScaling}
 Q_{n,d}(\theta)=\frac{q_{n,d}(\cos\theta)}{q_{n,d}(1)},
 \qquad F_{n,d}(z)=Q_{n,d}(z/n).
\end{equation}
The polynomial $q_{n,d}$ has the parity of $n$, by uniqueness and symmetry of the weight. Consequently $Q_{n,d}$ belongs to the space $\mathcal T_n'$ introduced earlier, whose frequencies are $-n,-n+2,\ldots,n-2,n$.

\begin{lemma}\label{lem:generalEndpointMinimum}
Among $T\in\mathcal T_n'$ with $T(0)=1$, the function $Q_{n,d}$ minimizes
\begin{equation}\label{eq:generalEndpointNorm}
 J_d(T)=\int_{-\pi/2}^{\pi/2}|T(\theta)|\,|\sin\theta|^{d-3}\,d\theta.
\end{equation}
Moreover, $\norm{Q_{n,d}}_\infty=1$ on the real line, and $|F_{n,d}(z)|\leq e^{|\operatorname{Im}z|}$ for every $z\in\C$.
\end{lemma}
\begin{proof}
Taking real parts and then averaging under $\theta\mapsto-\theta$ reduces the minimization to $T(\theta)=P(\cos\theta)$, where $P$ is a real polynomial of degree at most $n$ and parity $n$. The difference of two admissible polynomials vanishes at both $1$ and $-1$, so it has the form $(1-t^2)p(t)$ with $\deg p\leq n-2$. On changing variables in~\eqref{eq:generalEndpointNorm}, the measure becomes $(1-t^2)^{(d-4)/2}\,dt$. The additional factor $1-t^2$ in a variation restores the weight in~\eqref{eq:generalSignMoments}. Those moments therefore make the first variation vanish at $Q_{n,d}$, and convexity of the absolute value proves minimality. This also explains the exponent $d-3$ in the angular weight.

To prove the uniform bound when the weight is nonconstant, regard $q_{n,d}(e\cdot\omega)/q_{n,d}(1)$ as a zonal polynomial on $S^{d-2}$. The angular surface measure on this sphere is proportional to $\sin^{d-3}\theta\,d\theta$. Consider all spherical polynomials of degree at most $n$ and parity $n$. Averaging over the orthogonal transformations fixing $e$ preserves the value at $e$ and does not increase the integral norm; the average is a zonal polynomial of the form already considered. Thus our polynomial minimizes the integral norm in this larger space as well, subject to its value at $e$ being one. If its modulus exceeded one at another point, moving that point to $e$ by an orthogonal transformation and dividing by the resulting value at $e$ would decrease the norm. This contradiction proves the uniform bound. For $S^1$ the averaging includes reflection across the axis through $e$. Finally, the Bernstein bound for trigonometric polynomials, followed by the scaling in~\eqref{eq:generalEndpointScaling}, gives the stated bound on $\C$.
\end{proof}

After the substitution $x=n\theta$, the interval expands to the real line and the weight approaches $|x|^{d-3}$. This leads to the constant
\begin{equation}\label{eq:weightedEntireMinimum}
 M_d=\min_F\int_\R |F(x)|\,|x|^{d-3}\,dx,
 \qquad F(0)=1.
\end{equation}
Here the minimum is taken over real even entire functions of exponential type at most one for which the displayed integral is finite. Denote its minimizer by $\Phi_d$. The correspondence $F\mapsto F(|\cdot|)$ identifies this problem, up to the surface-area factor in polar coordinates, with the radial $L^1$ point-evaluation problem in $\R^{d-2}$. The existence and uniqueness of the radial extremizer follow from Dai, Gorbachev, and Tikhonov~\cite[Theorem~4.2]{DGT}. In particular, when $d=3$ the weight disappears: $\Phi_3=\varphi$ and $M_3=M$ in the notation of~\eqref{eq:HB}. Thus the H\"ormander--Bernhardsson problem is precisely the three-dimensional member of this family.

\begin{proposition}\label{prop:weightedEndpointConvergence}
For each fixed $d\geq3$, the functions $F_{n,d}$ converge to $\Phi_d$ locally uniformly on $\C$. For every real $b$ with $0\leq b\leq d-3$,
\begin{equation}\label{eq:weightedMomentConvergence}
 n^{b+1}\int_{-\pi/2}^{\pi/2}|Q_{n,d}(\theta)|\,|\sin\theta|^b\,d\theta
 \longrightarrow\int_\R|\Phi_d(x)|\,|x|^b\,dx.
\end{equation}
In particular, $n^{d-2}J_d(Q_{n,d})\longrightarrow M_d$.
\end{proposition}
\begin{proof}
Write $\alpha=d-3$. The proof follows the periodization argument used for $S^3$, but we keep the weight in order to control the tails. Let $F$ be real and even, of exponential type strictly less than one, and Schwartz on the real line, with $F(0)=1$. Its periodization
\[
 T_n(\theta)=\sum_{j\in\mathbb Z}(-1)^{nj}F\bigl(n(\theta+j\pi)\bigr)
\]
belongs to $\mathcal T_n'$ by Poisson summation, and $T_n(0)\to1$. The triangle inequality and $n|\sin(x/n)|\leq|x|$ give
\[
 n^{\alpha+1}J_d(T_n)\leq\int_\R|F(x)|\,|x|^\alpha\,dx.
\]
Such functions approximate $\Phi_d$ in the weighted integral norm. Indeed, choose a real even Schwartz function $\psi$ of type at most one with $\psi(0)=1$ and $|\psi|\leq1$, and use $\Phi_d((1-2\varepsilon)x)\psi(\varepsilon x)$. These functions have type less than one; dilation continuity and dominated convergence prove the approximation as $\varepsilon\downarrow0$. The extremal function belongs also to ordinary $L^1(\R)$, and its derivatives are bounded on the real line, so the products are Schwartz. Applying Lemma~\ref{lem:generalEndpointMinimum} to $T_n/T_n(0)$ proves the upper bound
\[
 \limsup_{n\to\infty}n^{\alpha+1}J_d(Q_{n,d})\leq M_d.
\]

The same lemma makes the functions $F_{n,d}$ a normal family. With $w_n(x)=n|\sin(x/n)|$, the norm being considered is
\[
 m_{n,d}:=n^{\alpha+1}J_d(Q_{n,d})
 =\int_{-\pi n/2}^{\pi n/2}|F_{n,d}(x)|\,w_n(x)^\alpha\,dx.
\]
Choose a subsequence realizing the lower limit of $m_{n,d}$ and then a locally uniformly convergent further subsequence. Its limit $F_*$ is real and even, has type at most one, and satisfies $F_*(0)=1$. Since $w_n(x)\to|x|$ locally, Fatou's lemma on bounded intervals, followed by expansion of the intervals, gives
\[
 M_d\leq\int_\R|F_*(x)|\,|x|^\alpha\,dx
 \leq\liminf_{n\to\infty}m_{n,d}.
\]
Together with the upper bound this proves convergence of the minimum values. Every locally uniform subsequential limit must now be a minimizer in~\eqref{eq:weightedEntireMinimum}; uniqueness gives convergence of the full sequence to $\Phi_d$.

Subtracting the integral over $[-R,R]$ from $m_{n,d}$ proves tightness of the mass with weight $w_n^\alpha$. On the expanding interval, $w_n(x)\geq2|x|/\pi$. For $b<\alpha$ and $|x|>R$, the ratio $w_n(x)^b/w_n(x)^\alpha$ is therefore at most $(\pi/(2R))^{\alpha-b}$. This bounds the tails for every lower weight in the proposition; local convergence supplies the integrals on $[-R,R]$. The case $b=\alpha$ follows from the already established mass convergence, proving~\eqref{eq:weightedMomentConvergence}.
\end{proof}

We will also need control of the normalized polynomial outside the interval of interpolation. This estimate makes the logarithmic part of the Green kernel negligible in the two even-dimensional cases considered below.

\begin{lemma}\label{lem:generalOutsideIntegral}
For each fixed $d\geq3$, as $n\to\infty$,
\begin{equation}\label{eq:generalOutsideIntegral}
 \int_1^\infty\frac{a\,q_{n,d}(1)}{q_{n,d}(a)}\,da=O(n^{-2}).
\end{equation}
\end{lemma}
\begin{proof}
All zeros of $F_{n,d}$ are real, since the zeros of $q_{n,d}$ lie in $(-1,1)$. Hurwitz's theorem and Proposition~\ref{prop:weightedEndpointConvergence} imply that all zeros of $\Phi_d$ are real as well. There are infinitely many: otherwise Hadamard factorization and evenness would make $\Phi_d$ a polynomial, contrary to its weighted integrability. Also $\Phi_d(0)=1$. Choose two positive zeros, counted with multiplicity. Local convergence supplies two corresponding roots $r_1,r_2$ of $q_{n,d}$ near one, for which $n\arccos r_j$ stays bounded.

Put $A(a)=q_{n,d}(a)/q_{n,d}(1)$ and let $\epsilon$ be the parity of $n$. Factoring over the positive roots $r_j$ gives
\[
 A(a)=a^\epsilon\prod_{j=1}^{\lfloor n/2\rfloor}
 \left(1+\frac{a^2-1}{1-r_j^2}\right).
\]
For $a=\sqrt{1+s^2/n^2}$, the two roots just selected show that $A(a)\geq(1+cs^2)^2$, with a fixed $c>0$, for all $s\geq0$ and all sufficiently large $n$. Since $a\,da=s\,ds/n^2$, the integral in~\eqref{eq:generalOutsideIntegral} is at most
\[
 \frac1{n^2}\int_0^\infty\frac{s\,ds}{(1+cs^2)^2}.
\]
This proves the estimate.
\end{proof}

\section{Bohr limits on \texorpdfstring{$S^4$ and $S^6$}{S4 and S6}}\label{sec:evenBohrLimits}

The limiting extremal polynomial problem is now available in every dimension. To obtain the Bohr limit, we must also control the interpolation remainder of the Green kernel. The explicit kernels~\eqref{eq:S4Green} and~\eqref{eq:S6Green} reduce this step to simple and double poles, with a logarithmic remainder that is of smaller order.

\begin{theorem}\label{thm:BohrLimitsFourSix}
With the extremal functions and minimum values defined in~\eqref{eq:weightedEntireMinimum},
\begin{equation}\label{eq:BohrLimitsFourSix}
 \begin{aligned}
 \lim_{n\to\infty}n^2H(n,S^4)&=\frac{M_4}{4},\\
 \lim_{n\to\infty}n^2H(n,S^6)
 &=\frac1{16}\left(2\int_\R|\Phi_6(x)|\,|x|\,dx-M_6\Phi_6''(0)\right).
 \end{aligned}
\end{equation}
\end{theorem}

The first formula uses only the minimum norm with weight $|x|$. In the second formula, $\Phi_6$ minimizes the norm with weight $|x|^3$, whereas the additional integral has weight $|x|$. Both limits are determined by their respective extremal functions, but the simple-pole structure in dimension four is what reduces the answer to a multiple of $M_4$ alone.

\begin{proof}
We separate the pole terms from the logarithm. Each nonconstant term of the kernels~\eqref{eq:S4Green} and~\eqref{eq:S6Green} has positive derivatives of every positive order. Its interpolation remainder therefore has the sign of $q_{n,d}$. Since the coefficients are positive, the integral errors of the terms add exactly.

For either dimension write $A(t)=q_{n,d}(t)/q_{n,d}(1)$, and let $I$ denote interpolation at the zeros of $q_{n,d}$. For $a>1$, polynomial division gives
\begin{equation}\label{eq:generalResolventRemainder}
 \frac1{a-t}-I\left[\frac1{a-t}\right]
 =\frac{A(t)}{(a-t)A(a)}.
\end{equation}
Integrating this identity in $a$ gives the remainder for $-\log((1-t)/2)$:
\begin{equation}\label{eq:generalLogRemainder}
 A(t)\int_1^\infty\frac{da}{(a-t)A(a)}.
\end{equation}
Indeed, the logarithm is the integral of $(a-t)^{-1}-(a+1)^{-1}$ from one to infinity, and the second term is constant in $t$ and is reproduced by $I$.

For $d=4$, setting $a=1$ in~\eqref{eq:generalResolventRemainder} shows that the pole contributes
\[
 \frac14\int_{-1}^1|A(t)|(1+t)\,dt
 =\frac14\int_{-1}^1|A(t)|\,dt
 =\frac14J_4(Q_{n,4}).
\]
Here $|A|$ is even. By Proposition~\ref{prop:weightedEndpointConvergence}, this contribution, multiplied by $n^2$, tends to $M_4/4$.

To estimate the logarithmic contribution $L_{n,4}$, pair $t$ and $-t$ in~\eqref{eq:generalLogRemainder}. Since $A(a)>0$ for $a\geq1$, this gives
\[
 \begin{aligned}
 L_{n,4}
 &=\frac12\int_0^1|A(t)|(1-t^2)
       \int_1^\infty\frac{a\,da}{(a^2-t^2)A(a)}\,dt\\
 &\leq\frac12\left(\int_0^1|A(t)|\,dt\right)
       \left(\int_1^\infty\frac{a\,da}{A(a)}\right).
 \end{aligned}
\]
We used $(1-t^2)/(a^2-t^2)\leq1$. The first factor is $O(n^{-2})$ by~\eqref{eq:weightedMomentConvergence}, and the second is $O(n^{-2})$ by Lemma~\ref{lem:generalOutsideIntegral}. Thus $L_{n,4}=O(n^{-4})$, proving the first limit.

For $d=6$, put $\lambda_n=A'(1)$ and abbreviate
\[
 J_{0,n}=\int_{-1}^1|A(t)|\,dt,
 \qquad J_{1,n}=\int_{-1}^1|A(t)|(1-t^2)\,dt.
\]
These integrals correspond to the angular weights $|\sin\theta|$ and $|\sin\theta|^3$. Differentiating~\eqref{eq:generalResolventRemainder} with respect to $a$, changing the sign, and setting $a=1$ yields
\[
 \frac1{(1-t)^2}-I\left[\frac1{(1-t)^2}\right]
 =A(t)\left(\frac1{(1-t)^2}+\frac{\lambda_n}{1-t}\right).
\]
Multiply by the weight in~\eqref{eq:S6Green} and pair opposite points. The double pole contributes $[2J_{0,n}+(\lambda_n-1)J_{1,n}]/16$, and the simple pole contributes $3J_{1,n}/16$. Their sum is therefore
\begin{equation}\label{eq:S6PoleContribution}
 \frac{2J_{0,n}+(\lambda_n+2)J_{1,n}}{16}.
\end{equation}
Proposition~\ref{prop:weightedEndpointConvergence}, together with local convergence of derivatives, gives
\[
 n^2J_{0,n}\longrightarrow\int_\R|\Phi_6(x)|\,|x|\,dx,
 \qquad n^4J_{1,n}\longrightarrow M_6,
 \qquad \frac{\lambda_n}{n^2}=-F_{n,6}''(0)\longrightarrow-\Phi_6''(0).
\]
These are precisely the terms in the second limit of~\eqref{eq:BohrLimitsFourSix}.

Finally, the logarithmic contribution on $S^6$ satisfies
\[
 0\leq L_{n,6}\leq\frac38
 \left(\int_0^1|A(t)|(1-t^2)\,dt\right)
 \left(\int_1^\infty\frac{a\,da}{A(a)}\right).
\]
The estimate follows by the same pairing as above, now with the weight $(1-t^2)^2$. Its factors are $O(n^{-4})$ and $O(n^{-2})$, so $L_{n,6}=O(n^{-6})$. It does not contribute to the limit, completing the proof.
\end{proof}

The simple pole in dimension four leads to the single minimum value $M_4$. The double pole in dimension six also retains a lower weighted moment and the second derivative of $\Phi_6$ at zero. Higher even dimensions involve further pole remainders. For $S^4$ and $S^6$ we can go further: the next section proves that the limits just obtained are also the best coefficients valid at every $n$.

\section{Uniform Bohr bounds on \texorpdfstring{$S^4$ and $S^6$}{S4 and S6}}\label{sec:evenBohrUniform}

The preceding limits give necessary lower bounds for the coefficients valid at every $n$. We now prove that those limiting values are sufficient. The idea is to move a fixed extremal source to larger harmonic gaps while keeping its norm at most one. A fractional-linear change of variable, followed by an angular dilation, gives the required comparison for the complete Green kernel.

Write
\[
 \beta_4=\frac{M_4}{4},\qquad
 \beta_6=\frac1{16}\left(2\int_\R|\Phi_6(x)|\,|x|\,dx-M_6\Phi_6''(0)\right).
\]
These are the limits established in~\eqref{eq:BohrLimitsFourSix}.

\begin{theorem}\label{thm:evenBohrUniform}
For $d=4$ and $d=6$,
\begin{equation}\label{eq:evenBohrSupremum}
 H(S^d)=\sup_{n\geq1}n^2H(n,S^d)
 =\lim_{n\to\infty}n^2H(n,S^d)=\beta_d.
\end{equation}
For every positive integer $n$, the finite value $n^2H(n,S^d)$ is strictly smaller than $\beta_d$. On $S^4$ we also have the quantitative estimate
\begin{equation}\label{eq:S4QuantitativeBohr}
 n^2H(n,S^4)\leq\beta_4\frac{(n+1)^2}{n(n+3)}.
\end{equation}
Consequently, every smooth $f\perp\mathcal H_{<n}$ on either sphere satisfies $\norm{f}_\infty\leq\beta_d n^{-2}\norm{\Delta f}_\infty$, and the coefficient $\beta_d$ cannot be decreased uniformly in $n$.
\end{theorem}

We first construct the sources used in the proof. The construction preserves the required moments in every dimension; the subsequent comparison of Green potentials is specific to dimensions four and six.

\begin{lemma}\label{lem:gapTransfer}
Fix integers $d\geq2$ and $n\geq1$, and put $s_n(t)=\sgn q_{n,d}(t)$. For an even integer $m\geq2$, set
\begin{equation}\label{eq:transferParameters}
 \varepsilon=m^{-2},\qquad c=1-\varepsilon,\qquad
 b=\frac{(d-1)c}{3(n+d-1)-(d-1)c}.
\end{equation}
Define
\begin{equation}\label{eq:mobiusSource}
 h_b(t)=\left(\frac{1+b}{1+bt}\right)^{n+d-1}
 s_n\!\left(\frac{t+b}{1+bt}\right),
\end{equation}
and, for $0<\theta<\pi$, define the zonal source
\begin{equation}\label{eq:dilatedSource}
 g_{m,n}(\cos\theta)=h_b(\cos(m\theta))
 \left(\frac{|\sin(m\theta)|}{m\sin\theta}\right)^{d-1}.
\end{equation}
Then $\norm{g_{m,n}}_\infty\leq1$, and $g_{m,n}\perp\mathcal H_{<mn}$ on $S^d$.
\end{lemma}
\begin{proof}
First we check the moments before dilation. Under the substitution $t=(s-b)/(1-bs)$, formula~\eqref{eq:mobiusSource} gives
\[
 h_b(t)\,d\mu_d(t)
 =C_b(1-bs)^{n-1}s_n(s)\,d\mu_d(s),\qquad
 C_b=\frac{(1+b)^{d/2}}{(1-b)^{n+d/2-1}}.
\]
For a polynomial $p$ of degree less than $n$, the expression $(1-bs)^{n-1}p((s-b)/(1-bs))$ is again a polynomial of degree less than $n$. Thus~\eqref{eq:generalSignMoments} shows that $h_b$ has gap $n$. Equivalently, the even periodic function $W(\theta)=h_b(\cos\theta)|\sin\theta|^{d-1}$ has no cosine frequencies below $n$. Multiplication of~\eqref{eq:dilatedSource} by $\sin^{d-1}\theta$ gives $m^{-(d-1)}W(m\theta)$, whose frequencies are at least $mn$. Since every polynomial of degree less than $mn$ in $\cos\theta$ is a cosine polynomial of that degree, this proves the spherical gap.

The function $h_b$ itself need not be bounded by one; its increase is compensated by the dilation factor. Put $x=\arccos(\cos(m\theta))$. Because $m$ is even, $x/m\leq\theta\leq\pi-x/m$, and hence $\sin\theta\geq\sin(x/m)$. The finite cosine identity gives
\[
 \frac{\sin x}{m\sin(x/m)}
 =\frac1m\sum_{j=0}^{m-1}\cos\left(\left(1-\frac{2j+1}{m}\right)x\right)
 \leq1-\frac c3(1-\cos x).
\]
Indeed, concavity of sine on $[0,\pi/2]$ implies $\cos(ax)\leq1-a^2(1-\cos x)$ when $|a|\leq1$ and $0\leq x\leq\pi$, and the average of the squared coefficients in this sum is $c/3$. Write $N=n+d-1$ and $u=1-\cos x$. Bernoulli's inequality, with exponent $N/(d-1)\geq1$, now yields
\[
 \left(1-\frac{cu}{3}\right)^{d-1}
 \leq\left(1-\frac{(d-1)cu}{3N}\right)^N
 =\left(\frac{1+b\cos x}{1+b}\right)^N.
\]
This is precisely the reciprocal of the amplitude in~\eqref{eq:mobiusSource}. The product in~\eqref{eq:dilatedSource} is therefore bounded by one. Endpoint values, which do not affect the source as an $L^\infty$ function, may be assigned arbitrarily.
\end{proof}

Let $f_{m,n}$ be the mean-zero Green potential of this source. Lemma~\ref{lem:gapTransfer} and the sharp kernel estimate give
\begin{equation}\label{eq:transferLowerBound}
 H(mn,S^d)\geq f_{m,n}(e).
\end{equation}
As in the proof of Theorem~\ref{thm:BohrAllDimensions}, heat regularization justifies this lower bound in the smooth class. It remains to evaluate the right-hand side without losing the lower-order terms of the Green kernel.

Periodic folding in the angular variable gives
\begin{equation}\label{eq:foldedGreen}
 \begin{aligned}
 m^2f_{m,n}(e)&=\int_{-1}^1 F_{m,d}(t)h_b(t)\,d\mu_d(t),\\
 F_{m,d}(\cos\theta)&=m^{2-d}\sum_{j=0}^{m-1}
 G_d\!\left(\cos\frac{\theta+2\pi j}{m}\right).
 \end{aligned}
\end{equation}
To see the factor, write the potential as half the integral over a full period, insert~\eqref{eq:dilatedSource}, and put $u=m\theta$. The resulting integral covers $m$ periods; the Jacobian, the dilation factor, and multiplication by $m^2$ together give $m^{2-d}$.

For brevity put $P_j(t)=(1-t)^{-j}$ and $L(t)=-\log((1-t)/2)$. The sums needed in~\eqref{eq:foldedGreen} are
\begin{equation}\label{eq:foldedElementaryKernels}
 \begin{aligned}
 \sum_{j=0}^{m-1}P_1\!\left(\cos\frac{\theta+2\pi j}{m}\right)
 &=m^2P_1(\cos\theta),\\
 \sum_{j=0}^{m-1}P_2\!\left(\cos\frac{\theta+2\pi j}{m}\right)
 &=m^4P_2(\cos\theta)-\frac{m^2(m^2-1)}3P_1(\cos\theta),\\
 \sum_{j=0}^{m-1}L\!\left(\cos\frac{\theta+2\pi j}{m}\right)
 &=L(\cos\theta)+(m-1)\log4.
 \end{aligned}
\end{equation}
The first identity follows by differentiating the cotangent sum, and the last follows from the product formula for sines. To obtain the second, differentiate the first twice and use $\partial_\theta^2P_1(\cos\theta)=3P_2(\cos\theta)-P_1(\cos\theta)$. The constant terms in the folded kernels disappear against $h_b$, whose mean is zero.

We shall also use one comparison of interpolation errors. In either dimension write $A(t)=q_{n,d}(t)/q_{n,d}(1)$ and set
\[
 E_j=\int_{-1}^1P_j(t)s_n(t)\,d\mu_d(t),\qquad
 E_L=\int_{-1}^1L(t)s_n(t)\,d\mu_d(t).
\]
These quantities are positive whenever the corresponding pole is integrable. Formulas~\eqref{eq:generalResolventRemainder} and~\eqref{eq:generalLogRemainder}, followed by pairing $t$ and $-t$, give
\[
 E_1=\int_{-1}^1\frac{|A(t)|}{1-t^2}\,d\mu_d(t),\qquad
 E_L=\int_{-1}^1|A(t)|\int_1^\infty
 \frac{a\,da}{(a^2-t^2)A(a)}\,d\mu_d(t).
\]
For $a\geq1$, root factorization and symmetry give $A(a)\geq a^n$. Also $a(1-t^2)/(a^2-t^2)\leq1/a$. Consequently
\begin{equation}\label{eq:logPoleComparison}
 E_L\leq E_1\int_1^\infty\frac{da}{aA(a)}\leq\frac{E_1}{n}.
\end{equation}

\begin{proof}[Proof of Theorem~\ref{thm:evenBohrUniform}]
We first treat $S^4$. Put $R_n=E_1/3$ and $L_n=E_L/3$. From~\eqref{eq:S4Green} and~\eqref{eq:logPoleComparison},
\begin{equation}\label{eq:S4PoleLogComparison}
 H(n,S^4)=R_n+L_n,\qquad
 R_n=\frac14\int_{-1}^1|A(t)|\,dt,\qquad
 0<L_n\leq\frac{R_n}{n}.
\end{equation}
The parameters in~\eqref{eq:transferParameters} become $b=(1-\varepsilon)/(n+2+\varepsilon)$. Set $\rho=(1+b)/(1-b)=(n+3)/(n+1+2\varepsilon)$. Under the fractional-linear substitution used above, polynomial division gives
\[
 \int_{-1}^1P_1(t)h_b(t)\,d\mu_4(t)=\rho E_1.
\]
Indeed, modulo a polynomial of degree less than $n$, the transformed numerator satisfies $(1-bs)^n/(1-s)=(1-b)^n/(1-s)$. The omitted polynomial integrates to zero against $s_n$.

The transformed logarithmic integral is positive as well. Apart from a polynomial of degree less than $n$, its kernel after substitution is
\[
 (1-bs)^{n-1}\log\frac{1-bs}{1-s}.
\]
Its $n$th derivative is
\[
 n!\int_b^1\frac{(a-b)^{n-1}}{(1-as)^{n+1}}\,da>0,
 \qquad -1<s<1.
\]
This follows by integrating $s/(1-as)$ in $a$ and differentiating after polynomial division. Interpolation at the zeros of $q_{n,4}$ therefore has remainder sign $s_n$, and the moment conditions give $\int Lh_b\,d\mu_4>0$.

The complete folded kernel from~\eqref{eq:foldedElementaryKernels} is $(P_1+\varepsilon L)/3$, up to a constant. Hence
\begin{equation}\label{eq:S4TransferredPotential}
 m^2f_{m,n}(e)=\rho R_n+\frac{\varepsilon}{3}\int_{-1}^1L(t)h_b(t)\,d\mu_4(t)>\rho R_n.
\end{equation}
For $n\geq2$ and even $m\geq4$, we have $\varepsilon\leq1/16$ and
\[
 \rho-\left(1+\frac1n\right)
 =\frac{n-1-2\varepsilon(n+1)}{n(n+1+2\varepsilon)}>0.
\]
Together with~\eqref{eq:S4PoleLogComparison}, this proves $m^2f_{m,n}(e)>H(n,S^4)$. For $n=1$, use $R_1=1/4$ and $H(1,S^4)=1/6+(\log2)/3<2/5$, whereas $\rho R_1\geq8/17>2/5$. Thus
\[
 m^2H(mn,S^4)>H(n,S^4)
\]
for every positive integer $n$ and every even integer $m\geq4$.

To obtain the quantitative estimate, fix $n$ and let $m$ tend to infinity through even integers in~\eqref{eq:S4TransferredPotential}. The functions $h_b$ stay uniformly bounded as $b\to1/(n+2)$, so the logarithmic term tends to zero. The already established limit~\eqref{eq:BohrLimitsFourSix} and~\eqref{eq:transferLowerBound} yield
\[
 \beta_4\geq n^2\frac{n+3}{n+1}R_n.
\]
Combining this with~\eqref{eq:S4PoleLogComparison} proves~\eqref{eq:S4QuantitativeBohr}. For $n\geq2$ its factor is strictly less than one. For $n=1$ the last displayed bound gives $\beta_4\geq1/2>H(1,S^4)$. Thus every finite rescaled constant is strictly below the limit, and their supremum equals that limit.

We now turn to $S^6$. Its kernel~\eqref{eq:S6Green} gives
\begin{equation}\label{eq:S6ErrorDecomposition}
 H(n,S^6)=\frac{E_2}{15}+\frac{E_1}{5}+\frac{E_L}{5}.
\end{equation}
In this case $b=5(1-\varepsilon)/(3(n+5)-5(1-\varepsilon))$, and the folded kernel is
\[
 F_{m,6}=\frac{P_2}{15}+\frac{10\varepsilon-1}{45}P_1
 +\frac{\varepsilon^2}{5}L+\text{constant}.
\]
For fixed $n$, let $m$ tend to infinity through even integers and put
\[
 b\longrightarrow\frac5{3n+10},\qquad
 x=\frac{b}{1-b}\longrightarrow\frac5{3n+5},\qquad
 \rho=\frac{1+b}{1-b}=1+2x.
\]
The same polynomial division, now keeping two pole terms, gives
\[
 \int P_1h_b\,d\mu_6=\rho^2E_1,\qquad
 \int P_2h_b\,d\mu_6=\rho E_2+\rho(n+1)xE_1.
\]
For the second identity expand $(1-bs)^{n+1}/(1-s)^2$ at $s=1$; the remainder is a polynomial of degree less than $n$ and disappears against $s_n$. Since $h_b$ remains uniformly bounded for fixed $n$, passage to the limit in~\eqref{eq:foldedGreen} is justified. We obtain
\begin{equation}\label{eq:S6TransferredLimit}
 \begin{aligned}
 \lim_{\substack{m\to\infty\\ m\ \mathrm{even}}}m^2f_{m,n}(e)
 &=F_n,\\
 F_n&=\frac{\rho E_2}{15}
 +\left(\frac{\rho(n+1)x}{15}-\frac{\rho^2}{45}\right)E_1,
 \qquad x=\frac5{3n+5}.
 \end{aligned}
\end{equation}
We shall show that $F_n>H(n,S^6)$ for every $n$.

Let $\lambda=A'(1)$. Root factorization gives $\lambda\geq n$. The simple- and double-pole interpolation remainders used in~\eqref{eq:S6PoleContribution} give
\[
 \begin{aligned}
 E_1&=\frac{15}{16}\int_{-1}^1|A(t)|(1-t^2)\,dt,\\
 E_2&=\frac{15}{16}\int_{-1}^1|A(t)|
 \bigl(1+t^2+\lambda(1-t^2)\bigr)\,dt.
 \end{aligned}
\]
Consequently $E_2/E_1\geq n+1$, while~\eqref{eq:logPoleComparison} gives $E_L/E_1\leq1/n$. Since $\rho>1$, these inequalities have the required direction in the difference between~\eqref{eq:S6TransferredLimit} and~\eqref{eq:S6ErrorDecomposition}. For $n\geq3$ they yield
\[
 \frac{45(F_n-H(n,S^6))}{E_1}
 \geq\frac{45n^3+69n^2-345n-225}{n(3n+5)^2}>0.
\]
The numerator, with $k=n-3$, is $45k^3+474k^2+1284k+576$, proving its positivity throughout this range.

The first two degrees require slightly better pole comparisons. For $n=1$, $A(t)=t$ and the displayed integral formulas give $E_2/E_1=4$. Using this value and $E_L/E_1\leq1$ gives $45(F_1-H(1,S^6))/E_1\geq3/8>0$.

For $n=2$, write $A(t)=(t^2-r^2)/(1-r^2)$. The sign moment gives $\int_0^r(1-t^2)^2\,dt=4/15$. At $r=1/3$ the integral is $376/1215>4/15$, so $r<1/3$. Put
\[
 u=\frac{\int_{-1}^1t^2|A(t)|\,dt}{\int_{-1}^1|A(t)|\,dt}.
\]
Bounding the numerator below by its signed integral and the denominator above by the integral of $(t^2+r^2)/(1-r^2)$ gives
\[
 u\geq\frac{1/5-r^2/3}{1/3+r^2}\geq\frac{11}{30}>\frac13.
\]
Thus $E_2/E_1=\lambda+(1+u)/(1-u)\geq4$. Together with $E_L/E_1\leq1/2$, this gives $45(F_2-H(2,S^6))/E_1\geq381/242>0$. The strict comparison is proved for every $n$.

Finally,~\eqref{eq:transferLowerBound},~\eqref{eq:S6TransferredLimit}, and the limit already established in~\eqref{eq:BohrLimitsFourSix} imply
\[
 \beta_6\geq n^2F_n>n^2H(n,S^6)
\]
for every fixed $n$. Taking the supremum and using convergence to $\beta_6$ proves~\eqref{eq:evenBohrSupremum} for $d=6$, completing the proof.
\end{proof}

The comparison concerns each finite value and the limiting coefficient. It does not establish monotonicity between consecutive values of $n^2H(n,S^4)$ or $n^2H(n,S^6)$. For other dimensions, Theorem~\ref{thm:BohrAllDimensions} still gives the exact constant at every fixed $n$, and Proposition~\ref{prop:weightedEndpointConvergence} identifies the limiting extremal polynomial problem. Determining the Bohr limit and comparing it with the supremum in those dimensions require further work.

\section{The Favard constant on \texorpdfstring{$S^2$}{S2}}\label{sec:favardS2}

We now turn to best uniform approximation of functions on $S^2$. It is easy to see that
\begin{equation}\label{eq:S2FavardUpperBound}
 K(n,S^2)\leq H(n,S^2)
\end{equation}
for every positive integer $n$. Indeed, let $f$ be a smooth function whose Laplacian is bounded in $L^\infty$ by one. Subtracting its mean does not change the approximation error, so we may assume that $f$ has mean zero. Let $p_{n-1}$ be the best $L^1(dt/2)$ approximant to $G_2$ found in Section~\ref{sec:zonalapproximation}. Convolving $-\Delta f$ with this polynomial gives a spherical polynomial $t_{n-1}\in\mathcal H_{<n}$. The Green representation and the kernel estimate~\eqref{eq:kernelApproximationBound} yield
\[
 E_{n-1}(f)_{L^\infty(S^2)}
 \leq\norm{f-t_{n-1}}_\infty
 \leq\frac12\int_{-1}^1|G_2(t)-p_{n-1}(t)|\,dt
 =H(n,S^2).
\]
Taking the supremum over $f$ proves~\eqref{eq:S2FavardUpperBound}.

We have not determined the exact values of $K(n,S^2)$ for $n\geq2$. A lower bound, however, is provided by a simple family of measures: take $2n$ equally spaced points on an equator and give them alternating signs. The corresponding test functions are spherical analogues of the Euler spline, with the negative Laplacian equal alternately to $1$ and $-1$ in successive meridional sectors. Although we do not know whether these functions are extremal for the full approximation problem, they give the following asymptotic estimate:

\begin{equation}\label{eq:equatorialFavardLowerBound}
 \liminf_{n\to\infty}n^2K(n,S^2)\geq\frac{\pi^2}{8}.
\end{equation}

For the proof, we construct an admissible measure, find the sign of its Green potential, and use that sign as the negative Laplacian of a test function. The resulting lower bound can then be evaluated asymptotically.

Fix a positive integer $n$. For $j=0,\ldots,2n-1$, put $x_j=(\cos(j\pi/n),\sin(j\pi/n),0)$ and define
\[
 \nu_n=\frac1{2n}\sum_{j=0}^{2n-1}(-1)^j\delta_{x_j}.
\]
This measure has total variation one and annihilates $\mathcal H_{<n}$. Indeed, restriction to the equator turns a spherical polynomial of degree less than $n$ into a trigonometric polynomial of the same degree, and the alternating sum annihilates all its frequencies. It follows that the absolute value of integration against $\nu_n$ is a lower bound for the best approximation error.

Use longitude $\varphi$ and height $z$ on the sphere, so that normalized surface measure is $dz\,d\varphi/(4\pi)$. Put $q(z)=\exp(-\operatorname{arctanh}|z|)$ and $a=q(z)^n$. The Fourier expansion of the logarithm, followed by the alternating sum, gives the Green potential
\[
 U_{\nu_n}(z,\varphi)
 =\int_{S^2}G_2(x\cdot y)\,d\nu_n(x)
 =\frac1{2n}\log\frac{1+2a\cos(n\varphi)+a^2}{1-2a\cos(n\varphi)+a^2}.
\]
The formula holds away from the atoms, which do not affect the integral. Its sign is $h_n(\varphi)=\sgn\cos(n\varphi)$. In particular, $h_n$ has mean zero and is bounded by one.

Let $f_n$ be the mean-zero Green potential of $h_n$. This is a continuous function, and $-\Delta f_n=h_n$ in the weak sense. Symmetry of the Green kernel and Fubini's theorem give
\[
 \int_{S^2}f_n\,d\nu_n
 =\int_{S^2}h_nU_{\nu_n}\,d\sigma_2
 =\norm{U_{\nu_n}}_{L^1(S^2)}=:L_n.
\]
Thus $E_{n-1}(f_n)_{L^\infty(S^2)}\geq L_n$. To remain within the smooth class used in our definition, apply heat regularization to $f_n$. It preserves zero mean and the bound on its Laplacian, and converges uniformly to $f_n$. Consequently $K(n,S^2)\geq L_n$ also with that definition.

It remains to compute the asymptotic size of $L_n$. Expanding the logarithm and integrating in longitude against $\sgn\cos(n\varphi)$ yields
\begin{equation}\label{eq:equatorialLowerIntegral}
 L_n=\frac4{\pi n}\sum_{j=0}^\infty\frac{(-1)^j}{(2j+1)^2}
 \int_0^1q(z)^{n(2j+1)}\,dz.
\end{equation}
The series is absolutely integrable. For each positive odd integer $k$, the substitution $v=n\operatorname{arctanh}z$ gives
\[
 n\int_0^1q(z)^{nk}\,dz
 =\int_0^\infty\frac{e^{-kv}}{\cosh^2(v/n)}\,dv
 \longrightarrow\frac1k.
\]
The integral is bounded by $1/k$, so the summable series with terms $(2j+1)^{-3}$ justifies passage to the limit in~\eqref{eq:equatorialLowerIntegral}. We obtain
\[
 \lim_{n\to\infty}n^2L_n
 =\frac4\pi\sum_{j=0}^\infty\frac{(-1)^j}{(2j+1)^3}
 =\frac{\pi^2}{8}.
\]
The last value is the classical second-order Favard constant in~\eqref{eq:FavardConstants}. Since $K(n,S^2)\geq L_n$, this proves~\eqref{eq:equatorialFavardLowerBound}.

The functions used in this proof have a direct geometric description. In spherical coordinates, with colatitude $\theta$ and longitude $\varphi$, they are the mean-zero solutions of
\begin{equation}\label{eq:sphericalEulerSpline}
 -\Delta f_n(\theta,\varphi)=\sgn\cos(n\varphi).
\end{equation}
Their negative Laplacian is constant in each of the $2n$ equal meridional sectors and changes sign from one sector to the next. Odd reflection across the bounding meridians shows that $f_n$ vanishes there; rotational symmetry and continuity also force it to vanish at the poles. The maximum principle gives alternating positive and negative lobes. The points carrying $\nu_n$ lie at the centers of the equatorial arcs, where $f_n(x_j)=(-1)^jL_n$ by rotational symmetry and the pairing computed above. These are spherical analogues of the Euler spline used to obtain a lower bound; their extremality for the full class is not asserted. Smooth regularizations provide the admissible smooth test functions.

Figure~\ref{fig:sphericalEulerSpline} shows the case $n=3$. The level curves were computed from the cosine expansion of the solution of~\eqref{eq:sphericalEulerSpline}; the values are multiplied by $n^2$ for readability. In this case $L_3\approx0.117783035656$, and the alternating equatorial values of $g_3=9f_3$ are approximately $\pm1.060047321$.

\begin{figure}[htbp]
\centering
\includegraphics[width=7.83cm]{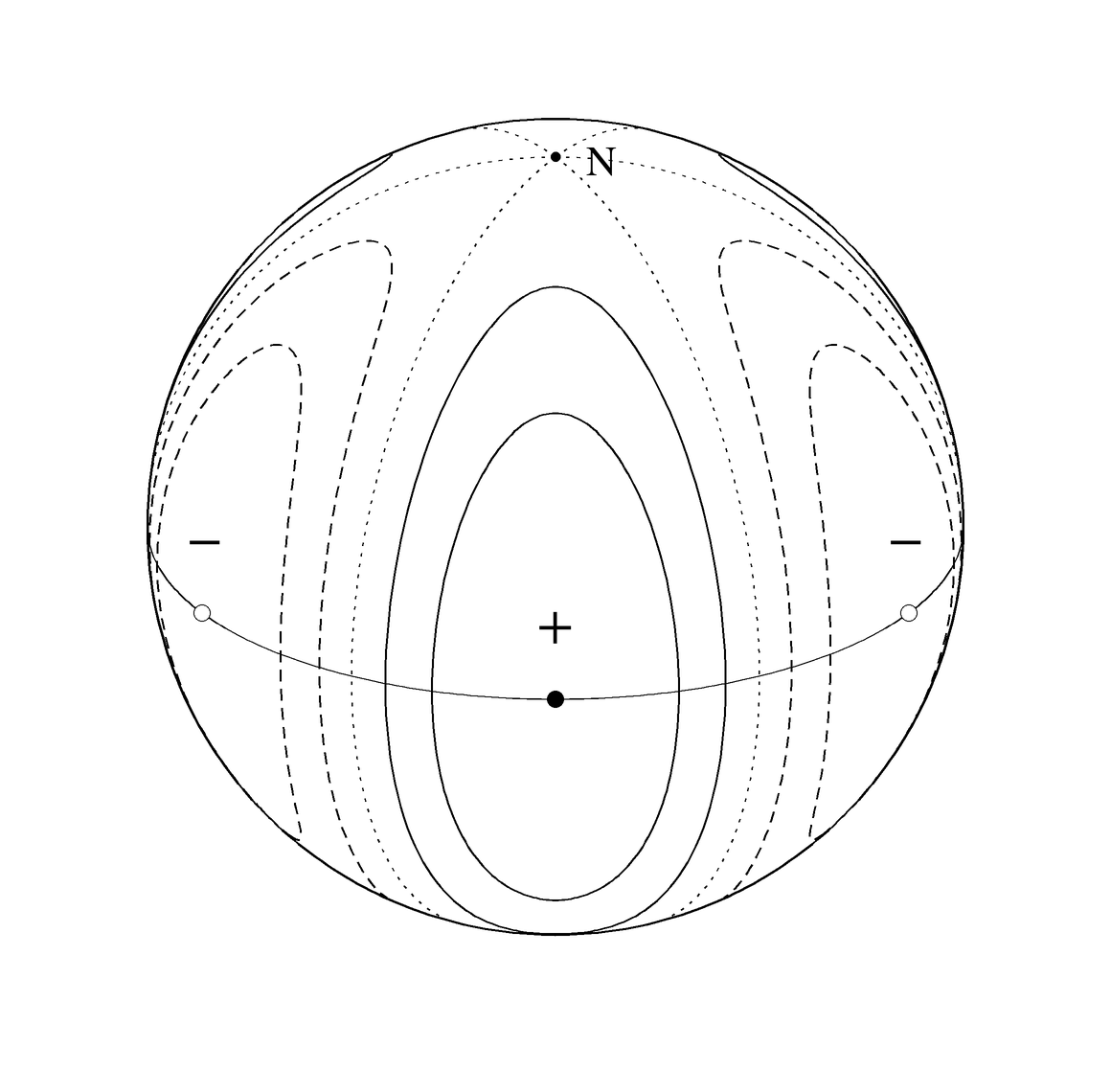}
\caption{The spherical Euler-spline analogue for $n=3$. Level curves of $g_3$ on the visible hemisphere, at levels $\pm0.35$ and $\pm0.70$. Positive levels are solid, negative levels dashed, and zero levels dotted.}
\label{fig:sphericalEulerSpline}
\end{figure}

As for an upper bound, no linear approximation method can improve on $H(n,S^2)$ uniformly over the whole class. The method obtained by interpolating the Green kernel is already optimal among all linear operators taking values in $\mathcal H_{<n}$. This follows from rotational symmetry.

\begin{theorem}\label{thm:optimalLinearS2}
For each positive integer $n$,
\begin{equation}\label{eq:optimalLinearS2}
 \inf_A\sup\left\{\norm{f-Af}_\infty:\;\norm{\Delta f}_\infty\leq1\right\}=H(n,S^2),
\end{equation}
where the supremum is taken over smooth complex-valued functions on $S^2$ and the infimum over linear operators $A:C^\infty(S^2)\to\mathcal H_{<n}$. The infimum is attained by the method obtained from the best $L^1$ approximant to $G_2$, with constants reproduced exactly.
\end{theorem}

\begin{proof}
We first show that an arbitrary linear method cannot have a smaller error. Let $M$ be the supremum in~\eqref{eq:optimalLinearS2} for a fixed operator $A$. We may assume that $M$ is finite. Then $A$ reproduces constants, since all their scalar multiples belong to the class. By scaling, $\norm{f-Af}_\infty\leq M\norm{\Delta f}_\infty$ for every smooth $f$. In particular, $\norm{Af}_\infty\leq\norm{f}_\infty+M\norm{\Delta f}_\infty$, so $A$ is continuous for the $C^2$ topology and the following average is well defined.

For a rotation $Q\in SO(3)$, write $U_Qf(x)=f(Q^{-1}x)$, and average the conjugate operators with respect to normalized Haar measure:
\[
 \widetilde A f=\int_{SO(3)}U_Q^{-1}AU_Qf\,dQ.
\]
Rotations preserve the class, the uniform norm, and $\mathcal H_{<n}$. Each conjugate therefore has the same maximal error $M$, and the triangle inequality gives $\norm{f-\widetilde Af}_\infty\leq M$ throughout the class. Thus averaging does not increase the maximal error.

The averaged operator commutes with rotations. Differentiating along the three one-parameter groups of coordinate rotations shows that it also commutes with $\Delta$, the sum of the squares of their infinitesimal generators. If $Y\in\mathcal H_\ell$ with $\ell\geq n$, then $\widetilde AY$ must have the same Laplacian eigenvalue $-\ell(\ell+1)$ as $Y$. But $\widetilde AY$ belongs to $\mathcal H_{<n}$, where that eigenvalue does not occur, so $\widetilde AY=0$. The spherical harmonic expansion of a smooth function converges in $C^2$; continuity therefore gives $\widetilde Af=0$ whenever $f\perp\mathcal H_{<n}$. On this subclass the approximation error is exactly $\norm{f}_\infty$. Taking its supremum gives $M\geq H(n,S^2)$ by definition~\eqref{eq:sphericalBohrFavard}.

To attain this bound, let $p_{n-1}$ be the best $L^1(dt/2)$ approximant to $G_2$ constructed in Section~\ref{sec:zonalapproximation}, and set
\[
 A_*f(x)=\bar f+\int_{S^2}p_{n-1}(x\cdot y)(-\Delta f)(y)\,d\sigma_2(y),
 \qquad \bar f=\int_{S^2}f\,d\sigma_2.
\]
This is a linear operator with values in $\mathcal H_{<n}$, and it reproduces constants. Applying~\eqref{eq:kernelApproximationBound} to $f-\bar f$ gives $\norm{f-A_*f}_\infty\leq H(n,S^2)$ throughout the class. Together with the lower bound just proved, this establishes~\eqref{eq:optimalLinearS2}.
\end{proof}

Thus, for each positive integer $n$, the polynomial interpolating the Green kernel $G_2$ at the zeros of $U_n$ gives an optimal linear approximation method among all linear methods with values in $\mathcal H_{<n}$. Its exact maximal error on the class is $H(n,S^2)$, and it yields the Favard-type upper bound~\eqref{eq:S2FavardUpperBound} for each $n$. Any improvement on this bound for the full approximation class must use a choice of the approximating polynomial that depends nonlinearly on the function.

Nevertheless, it is somewhat surprising that best approximation gives a strictly smaller constant at every degree $n\geq2$.

\begin{theorem}\label{thm:strictFavardS2}
For each integer $n\geq2$,
\begin{equation}\label{eq:strictFavardS2}
 K(n,S^2)<H(n,S^2).
\end{equation}
\end{theorem}

\begin{proof}
We use measures annihilating the approximating polynomials. Equality in the kernel estimate would require the errors with different poles to agree in phase; the annihilation conditions will make this impossible. Compactness then turns the exclusion of equality into a strict bound for the whole class.

Fix $n$, let $p_{n-1}$ be the interpolating polynomial for $G_2$, and put $R_n=G_2-p_{n-1}$. The kernel computations above and rotational invariance give $\int_{S^2}|R_n(x\cdot y)|\,d\sigma_2(y)=H(n,S^2)$ for every pole $x$. Moreover, $\sgn R_n(t)=\sgn U_n(t)$ almost everywhere, so the sign changes precisely at $t_j=\cos(j\pi/(n+1))$, for $j=1,\ldots,n$.

Let $\mathcal M_n$ be the set of finite complex Borel measures on $S^2$ with total variation at most one which annihilate $\mathcal H_{<n}$. For such a measure define
\[
 U_\mu(y)=\int_{S^2}R_n(x\cdot y)\,d\mu(x).
\]
Annihilation of the polynomial part means that this is also the Green potential of $\mu$. The Green representation and Fubini's theorem give $\int f\,d\mu=\int(-\Delta f)U_\mu\,d\sigma_2$ for every smooth $f$. The dual formula for best approximation, followed by the triangle inequality for the potential, therefore yields
\begin{equation}\label{eq:S2dualResidualBound}
 K(n,S^2)\leq\sup_{\mu\in\mathcal M_n}\norm{U_\mu}_1\leq H(n,S^2).
\end{equation}
We shall prove that the second inequality is strict.

First, the supremum over measures is attained. The set $\mathcal M_n$ is weak-* compact. Approximate $G_2$ in $L^1(dt/2)$ by continuous functions. For measures of total variation at most one, the error in the corresponding potentials is bounded in $L^1(S^2)$ by the same kernel error, uniformly in the measure. For a continuous kernel, weak-* convergence of a sequence of measures gives pointwise convergence of their potentials, and their uniform boundedness gives convergence in $L^1$. Approximation of $G_2$ thus shows that $\mu\mapsto U_\mu$ is weak-* to $L^1$ continuous on the measure unit ball. It follows that $\norm{U_\mu}_1$ attains its maximum on $\mathcal M_n$.

Suppose that the maximum equals $H(n,S^2)$. The maximizing measure then has total variation one. Write $d\mu=a\,d\nu$, where $\nu=|\mu|$ and $|a|=1$ almost everywhere with respect to $\nu$. Equality in the integrated triangle inequality forces all the contributions to have a common phase for almost every $y$. Thus there is a measurable function $b$, with $|b|=1$ almost everywhere, such that
\begin{equation}\label{eq:S2phaseAgreement}
 a(x)\,\sgn R_n(x\cdot y)=b(y)
\end{equation}
for almost every pair $(x,y)$ with respect to $\nu\otimes\sigma_2$. The residual is nonzero almost everywhere for each pole. By Fubini's theorem we may therefore fix a pole $e$ for which this identity holds for almost every $y$. For $\nu$-almost every other pole $x$, the two real sign functions $\sgn R_n(x\cdot y)$ and $\sgn R_n(e\cdot y)$ must agree or be opposite almost everywhere.

This is possible only when $x=e$ or $x=-e$. Indeed, the sign-change curves for a pole $x$ are the circles $x\cdot y=t_j$. If $x$ and $e$ are not parallel, each such circle meets each circle $e\cdot y=t_k$ in at most two points. We can therefore choose a point on a sign-change circle for $x$ which lies on none of those for $e$. In a small neighborhood of that point the sign associated with $e$ is constant, whereas the sign associated with $x$ changes across the circle. Both sides have positive surface measure, contradicting agreement, or opposite agreement, almost everywhere.

Consequently $\mu=A\delta_e+B\delta_{-e}$ for some complex numbers $A$ and $B$. Since $n\geq2$, the measure annihilates both the constant function and the linear function $x\mapsto e\cdot x$. These two conditions give $A+B=0$ and $A-B=0$, hence $\mu=0$. This contradicts its total variation being one. The maximum in~\eqref{eq:S2dualResidualBound} is therefore strictly smaller than $H(n,S^2)$, proving~\eqref{eq:strictFavardS2}.
\end{proof}

Thus the coincidence of the two constants in the classical periodic problem fails on $S^2$ for every $n\geq2$. The compactness argument gives a positive gap for each fixed $n$; it does not show that this gap remains positive uniformly in $n$ after multiplication by $n^2$.

\section{An algorithm for approximating smooth functions on \texorpdfstring{$S^2$}{S2}}\label{sec:constructiveS2}

The strict inequality just proved does not give a procedure for finding a best approximant to a prescribed function. Theorem~\ref{thm:optimalLinearS2}, however, provides an explicit method which is optimal among all linear methods. We describe its computation and apply it to a function arising from a geometric model used in geodesy.

Let $f$ be a smooth function on $S^2$, and choose a positive integer $n$. Recall that $G_2(t)=-1-\log((1-t)/2)$ and that its best $L^1(dt/2)$ approximant $p_{n-1}$ interpolates it at $t_j=\cos(j\pi/(n+1))$, for $j=1,\ldots,n$. Thus the polynomial needed for the method is obtained from $n$ values of a known function. For example, it can be written in Lagrange form as
\[
 p_{n-1}(t)=\sum_{j=1}^n G_2(t_j)\frac{U_n(t)}{(t-t_j)U_n'(t_j)}.
\]
The quotient in each summand is a polynomial, with its value at $t_j$ understood by continuity. Restoring the mean of $f$, we obtain the spherical polynomial
\begin{equation}\label{eq:constructiveGreenMethod}
 t_{n-1}(\omega)=A_nf(\omega)
 =\bar f+\int_{S^2}p_{n-1}(\omega\cdot\eta)(-\Delta f)(\eta)\,d\sigma_2(\eta).
\end{equation}
It belongs to $\mathcal H_{<n}$ and, by~\eqref{eq:H22fixed} and~\eqref{eq:H22}, satisfies
\begin{equation}\label{eq:constructiveGreenError}
 \norm{f-t_{n-1}}_\infty
 \leq H(n,S^2)\norm{\Delta f}_\infty
 \leq\frac{7\zeta(3)}{4n^2}\norm{\Delta f}_\infty.
\end{equation}
Optimality here concerns the maximal error over the whole class; the polynomial $A_nf$ need not be a best approximant to an individual function.

For numerical purposes we can avoid differentiating $f$. Self-adjointness of the spherical Laplacian transfers it to the polynomial in~\eqref{eq:constructiveGreenMethod}. Since the Laplacian of a zonal function $p(t)$ on $S^2$ is $(1-t^2)p''(t)-2tp'(t)$, define
\begin{equation}\label{eq:constructiveValueKernel}
 \mathcal L_n(t)=1+2tp_{n-1}'(t)-(1-t^2)p_{n-1}''(t).
\end{equation}
Then the same polynomial is given directly by the values of $f$:
\begin{equation}\label{eq:constructiveValues}
 t_{n-1}(\omega)=\int_{S^2}\mathcal L_n(\omega\cdot\eta)f(\eta)\,d\sigma_2(\eta).
\end{equation}
The constant term $1$ in~\eqref{eq:constructiveValueKernel} accounts for $\bar f$. In particular, this formula applies without a zero-mean assumption.

Here is a concrete quadrature procedure for computing this polynomial.
\begin{enumerate}
\item Interpolate $G_2$ at the nodes $t_j$ and compute $\mathcal L_n$ by~\eqref{eq:constructiveValueKernel}. This step depends only on $n$ and can be performed once for all functions to be approximated.
\item Choose an integer $q\geq n$. Let $u_i,w_i$, for $i=1,\ldots,q$, be the Gauss--Legendre nodes and weights on $[-1,1]$, with $\sum_iw_i=2$. Combine them with the $2q$ equally spaced longitudes $\varphi_j=j\pi/q$, for $j=0,\ldots,2q-1$, to form the spherical nodes
\[
 \eta_{ij}=\bigl(\sqrt{1-u_i^2}\cos\varphi_j,\,
 \sqrt{1-u_i^2}\sin\varphi_j,\,u_i\bigr).
\]
Evaluate $f$ at these $2q^2$ points.
\item Replace the integral in~\eqref{eq:constructiveValues} by the finite sum
\begin{equation}\label{eq:constructiveQuadrature}
 \widetilde t_{n-1,q}(\omega)
 =\frac1{4q}\sum_{i=1}^q\sum_{j=0}^{2q-1}
 w_i f(\eta_{ij})\mathcal L_n(\omega\cdot\eta_{ij}).
\end{equation}
This is already a spherical polynomial of degree at most $n-1$. To write it in Cartesian coordinates, expand $\mathcal L_n(\omega\cdot\eta_{ij})$ in powers of the coordinates of $\omega$ and collect the coefficients. Increase $q$ to check their numerical stability.
\end{enumerate}
The factor $1/(4q)$ corresponds to our normalized surface measure. Estimate~\eqref{eq:constructiveGreenError} bounds the error of the exact integral; for the computed polynomial one adds the quadrature error $\norm{A_nf-\widetilde t_{n-1,q}}_\infty$. Agreement under refinement is a numerical check, rather than a certified bound for this additional error.

As an example, consider the radial function of a triaxial ellipsoid. Such models are used in geodesy, and the expansion of their radial functions in spherical harmonics is discussed in~\cite[\S\,1.2]{HuShumBevis}. Choose the model semiaxes $1$, $0.9$, and $0.8$, and put
\begin{equation}\label{eq:ellipsoidExample}
 f(x,y,z)=\left(x^2+\frac{y^2}{0.9^2}+\frac{z^2}{0.8^2}\right)^{-1/2},
 \qquad x^2+y^2+z^2=1.
\end{equation}
For a unit direction $(x,y,z)$, this is the distance from the origin to the ellipsoid along that ray. Thus the function is defined on the unit sphere, while its values describe a different surface. The semiaxes are illustrative parameters chosen to make the directional variation visible. The function is smooth, depends on all three coordinates, and has no axis of rotational symmetry.

Its Laplacian can be computed directly. Write $D=\operatorname{diag}(1,100/81,25/16)$, $\omega=(x,y,z)$, $s=\omega^TD\omega$, and $v=\omega^TD^2\omega$. Then
\[
 \Delta_{S^2}f=\frac{3v-(\operatorname{tr}D)s}{s^{5/2}},
 \qquad M:=\norm{\Delta_{S^2}f}_\infty=\frac{1033}{1296}\approx0.797067901.
\]
To check the norm, the inequalities $s\leq v\leq(1+25/16)s-25/16$ reduce the extrema to a calculation on $1\leq s\leq25/16$. The most negative value is $-M$, attained at $(\pm1,0,0)$; the positive maximum is $0.455901235\ldots$, attained at $(0,0,\pm1)$.

Take $n=5$, so that the approximating polynomial has degree at most four. The exact operator preserves evenness in each coordinate. On the sphere, replacing the constant and quadratic terms by homogeneous quartic expressions using $x^2+y^2+z^2=1$ gives a representation with only six monomials. Computation by~\eqref{eq:constructiveQuadrature} yields, with the coefficients rounded,
\begin{equation}\label{eq:ellipsoidPolynomial}
 \begin{aligned}
 t_4(x,y,z)\approx{}&0.973987668\,x^4+0.897782497\,y^4\\
 &+0.815228281\,z^4+1.866629219\,x^2y^2\\
 &+1.765700170\,x^2z^2+1.706320701\,y^2z^2.
 \end{aligned}
\end{equation}
This is a spherical polynomial containing harmonic degrees $0$, $2$, and $4$. It need not be a homogeneous harmonic polynomial of degree four.

Table~\ref{tab:ellipsoidApproximation} compares the computed errors for several small values of $n$ with the two bounds in~\eqref{eq:constructiveGreenError}. We used $q=64$ and checked the coefficients against $q=32$. For $n=5$, the six coefficients changed by less than $2\cdot10^{-14}$. The errors were sampled on the triangular mesh $x^2=i/2000$, $y^2=(1-x^2)j/2000$, $z^2=1-x^2-y^2$, for $i,j=0,\ldots,2000$. Evenness permits this reduction to one octant. The sampled errors are numerical values, not certified uniform maxima; the last two columns give numerical evaluations of the theoretical bounds for the exact operator.

\begin{table}[htbp]
\caption{Approximation of the radial function~\eqref{eq:ellipsoidExample}.}\label{tab:ellipsoidApproximation}
\centering
\begin{tabular}{crrr}
\hline
$n$ & Sampled error & $M H(n,S^2)$ & $7\zeta(3)M/(4n^2)$\\
\hline
2 & 0.1066903 & 0.2085321 & 0.4191779\\
3 & 0.0528494 & 0.1114651 & 0.1863013\\
4 & 0.0398255 & 0.0697369 & 0.1047945\\
5 & 0.0260123 & 0.0478461 & 0.0670685\\
\hline
\end{tabular}
\end{table}

For $n=5$, the largest sampled error occurs at $(\pm1,0,0)$ and is approximately $0.0260123$, whereas~\eqref{eq:constructiveGreenError} guarantees an error below $0.04785$ for the exact polynomial. The construction therefore gives an explicit polynomial with a uniform error bound, without solving a separate best-approximation problem for this function.

\section{An explicit Favard gap on \texorpdfstring{$S^2$}{S2}}\label{sec:explicitFavardGapS2}

We conclude the discussion of $S^2$ by making the strict inequality quantitative. The idea is to distinguish two classes of functions. If the Green interpolation method already gives a smaller error, we use it. Otherwise its error forces the Laplacian to be close to a rotated Bohr extremal source, whose potential admits a better approximant.

\begin{proposition}\label{prop:explicitFavardGapS2}
For every integer $n\geq32$,
\begin{equation}\label{eq:explicitFavardGapS2}
 K(n,S^2)\leq H(n,S^2)-\frac{0.000000238}{n^7}.
\end{equation}
For the concrete case $n=3$, we have
\begin{equation}\label{eq:explicitFavardGapThree}
 K(3,S^2)\leq H(3,S^2)-0.00000000144.
\end{equation}
Both estimates apply to complex-valued functions. The decimal constants are chosen below the bounds obtained in the proof, so the displayed inequalities do not depend on numerical rounding.
\end{proposition}

We first give a quantitative form of the two-class argument. Fix $n$, put $h=H(n,S^2)$, and let $p_{n-1}$ be the best approximant to $G_2$ already constructed. Write $R_n=G_2-p_{n-1}$ and, for a pole $e$, put
\[
 s_{n,e}(y)=\sgn U_n(e\cdot y),\qquad
 \Phi_{n,e}=G_2*s_{n,e}.
\]
Suppose that the best approximation error of this Bohr extremal is at most $h-b$, with $b>0$. Rotation invariance makes this condition independent of $e$. We claim that
\begin{equation}\label{eq:quantitativeGapTransfer}
 K(n,S^2)\leq h-\eta,\qquad
 \eta=\min\left\{\frac b2,\frac{h n^4b^4}{512(n+1)}\right\}.
\end{equation}

To prove the claim, we need to control the region where $R_n$ is small. Interpolating the integral representation of $-\log(1-t)$ gives
\[
 R_n(t)=U_n(t)J_n(t),\qquad
 J_n(t)=\int_1^\infty\frac{da}{(a-t)U_n(a)}.
\]
The factor $J_n$ is positive and increasing. Substitution $a=\cosh v$, followed by expansion of $1/\sinh((n+1)v)$, gives $J_n(-1)=h/(n+1)$ by~\eqref{eq:H22fixed}. Hence $|R_n(t)|\geq h|U_n(t)|/(n+1)$, and
\begin{equation}\label{eq:gapSublevelBound}
 \sigma_2\{y:|R_n(e\cdot y)|\leq u\}\leq\frac{(n+1)u}{h}.
\end{equation}
Here is a direct verification of the last step. Put $m=n+1$ and $v=mu/h$. Only $v<1$ needs consideration. Since $U_n(\cos\theta)=\sin(m\theta)/\sin\theta$, the set in question is contained in $|\sin(m\theta)|\leq v$. With $\alpha=\arcsin v$, its containing set has normalized measure
\[
 1-\cos(\alpha/m)+\cot\frac\pi{2m}\sin(\alpha/m).
\]
As a function of $\alpha$ on $[0,\pi/2]$, this expression is convex and has endpoint values zero and one. It is therefore at most $2\alpha/\pi\leq\sin\alpha=v$, proving~\eqref{eq:gapSublevelBound}.

Now let $g=-\Delta f$, where $\norm{g}_\infty\leq1$, and subtract the mean of $f$. If $\norm{R_n*g}_\infty\leq h-\eta$, the polynomial $p_{n-1}*g$ suffices. Otherwise choose a pole $e$ where the residual has maximal modulus and a unit complex number $\lambda$ with the same phase. For $d=g-\lambda s_{n,e}$, the elementary inequality $|z-1|^2\leq2(1-\operatorname{Re}z)$ on the unit disk gives
\[
 \int_{S^2}|R_n(e\cdot y)|\,|d(y)|^2\,d\sigma_2(y)<2\eta.
\]
Split the integral of $|d|^2$ at $|R_n|=u$, use $|d|\leq2$ on the first part, and apply~\eqref{eq:gapSublevelBound}. Optimizing in $u$ yields
\[
 \norm{d}_2\leq2\left(\frac{2(n+1)\eta}{h}\right)^{1/4}.
\]
The part $G_{\geq n}$ of the Green kernel containing degrees at least $n$ has squared integral norm
\[
 \norm{G_{\geq n}}_2^2
 =\sum_{\ell\geq n}\frac{2\ell+1}{\ell^2(\ell+1)^2}
 =\frac1{n^2}.
\]
Subtracting the lower-degree part of $G_2*d$ and applying Cauchy--Schwarz therefore gives
\[
 E_{n-1}(f)_{L^\infty(S^2)}
 \leq h-b+\frac2n\left(\frac{2(n+1)\eta}{h}\right)^{1/4}.
\]
The choice in~\eqref{eq:quantitativeGapTransfer} makes the last term at most $b/2$ and also ensures $\eta\leq b/2$. This proves the claim for the whole class.

We next estimate the approximation gain $b$. Write $\Phi_{n,e}(x)=\Phi_n(e\cdot x)$ and $\psi(\theta)=\Phi_n(\cos\theta)$. Put $a=\pi/(n+1)$ and $\tau=\tan(a/2)$. The zonal equation gives, on $ja<\theta<(j+1)a$,
\[
 \psi'(\theta)=(-1)^j\frac{\cos\theta-c_j}{\sin\theta},\qquad
 c_j=\frac{\cos((j+1/2)a)}{\cos(a/2)}.
\]
This derivative is monotone on each interval, and its values at the interior nodes alternate between $\tau$ and $-\tau$. Thus $|\psi'|\leq\tau$. On $[a,\pi-a]$, the equation $-\psi''-\cot\theta\,\psi'=\sgn\sin((n+1)\theta)$ bounds $|\psi''|$ by $1+\tau\cot a=(3-\tau^2)/2$. On the two end intervals the explicit formula bounds it by $(1+\tau^2)/2$, which is no larger.

The even periodic extension of $\psi$ has an absolutely continuous first derivative. Applying the classical Favard estimate~\eqref{eq:FavardApproximation} for the second derivative, and taking the even part of the approximating polynomial, gives
\begin{equation}\label{eq:BohrPrototypeApproximation}
 E_{n-1}(\Phi_{n,e})_{L^\infty(S^2)}
 \leq\frac{\pi^2}{16n^2}(3-\tau^2)
 <\frac{3\pi^2}{16n^2}.
\end{equation}
A cosine polynomial of degree at most $n-1$ is a polynomial in $e\cdot x$ of the same degree, so this is a spherical approximation estimate.

Each summand of $n^2H(n,S^2)$ in~\eqref{eq:H22fixed} increases with $n$. At $n=32$, its first eight summands already have sum greater than $3(22/7)^2/16+1/8$. This is a rational comparison; together with $\pi<22/7$, it shows that we may take $b=0.125/n^2$ for every $n\geq32$. Substitution in~\eqref{eq:quantitativeGapTransfer}, using $h\geq n^{-2}$ and $n+1\leq2n$, gives $\eta\geq0.000000238/n^7$. This proves~\eqref{eq:explicitFavardGapS2}.

For $n=3$, a concrete polynomial gives a better starting point. Put $a=1/\sqrt2$ and $c=\sqrt2-1$. The Bohr extremal is the odd zonal function $\Phi_{3,e}(x)=F(e\cdot x)$, where, on the positive half-interval,
\[
 F(t)=
 \begin{cases}
 -\tfrac12\log(1-t^2)-c\operatorname{arctanh}t,&0\leq t\leq a,\\
 F(a)+\log\dfrac{1+t}{1+a},&a\leq t\leq1.
 \end{cases}
\]
Its value at one is $h=H(3,S^2)$. The spherical polynomial $0.03125\,e\cdot x$ satisfies
\begin{equation}\label{eq:gapThreePrototype}
 \norm{\Phi_{3,e}-0.03125\,e\cdot x}_\infty=h-0.03125.
\end{equation}
For completeness, on $[0,a]$ we have $F'(t)=(t-c)/(1-t^2)$ and $F''(t)>0$; on $[a,1]$ we have $F'(t)=1/(1+t)>1/2$. Also
\[
 -F(c)=\int_0^c\frac{c-u}{1-u^2}\,du
 \leq\frac{c^2}{2}+\frac{c^3}{24}<0.09.
\]
The middle bound follows from $1-c^2=2c$ and $1/(1-u^2)\leq1+u^2/(1-c^2)$; the last follows from $c<5/12$. The minimum of $F(t)-0.03125t$ occurs at a point $r$ satisfying $r+0.03125r^2=c+0.03125$, so $r<0.45$. Its negative magnitude is therefore less than $0.1040625$. By the first two summands in~\eqref{eq:H22fixed}, $h>296/2145$, and $h-0.03125>0.1040625$. The positive maximum occurs at one, where its value is $h-0.03125$. Oddness proves~\eqref{eq:gapThreePrototype}.

We may consequently use $b=0.03125$ in~\eqref{eq:quantitativeGapTransfer}. With $n=3$ and $h>296/2145$, both entries in that minimum exceed $0.00000000144$. This proves~\eqref{eq:explicitFavardGapThree}. The concrete polynomial was used for the model extremal, while the two-class argument transfers its gain to every admissible function, including complex-valued ones.

Unfortunately, we have not been able to determine the exact values of $K(n,S^2)$ for $n\geq2$, or to establish a positive gap that survives multiplication by $n^2$ and passage to the limit. The guaranteed gain in~\eqref{eq:explicitFavardGapS2} then tends to zero. In particular, the question whether $\limsup_{n\to\infty}n^2K(n,S^2)<7\zeta(3)/4$ remains open in this work.

\section{The Favard constant on \texorpdfstring{$S^3$}{S3}}\label{sec:favardconstants}

We now give the exact dual formulation in both dimensions and extend Theorem~\ref{thm:strictFavardS2} to $S^3$. The function $f$ need not have missing low harmonics: these are available to the approximating polynomial. As explained in Section~\ref{sec:greenrepresentation}, we may still assume that $f$ has mean zero.

There is an exact dual formulation of this problem in terms of measures. Fix $d=2$ or $d=3$ and a positive integer $n$. Let $\mathcal M_n$ be the set of finite complex Borel measures on $S^d$ with total variation at most one which annihilate $\mathcal H_{<n}$. For each such measure, define its Green potential by
\[
 U_\mu(y)=\int_{S^d}G_d(x\cdot y)\,d\mu(x).
\]
The integrability of $G_d$ and rotational invariance imply that $U_\mu$ belongs to $L^1(S^d)$. The Favard constant is given exactly by
\begin{equation}\label{eq:exactFavardDuality}
 K(n,S^d)=\sup_{\mu\in\mathcal M_n}\inf_{c\in\C}
 \norm{U_\mu-c}_{L^1(S^d)}.
\end{equation}
Thus an upper bound for $K(n,S^d)$ is equivalent to a uniform bound for the best $L^1$ approximation of these potentials by constants. The approximating constant may depend on the measure.

The subtraction of a constant reflects the fact that $g=-\Delta f$ has mean zero. Indeed, the Green representation gives
\[
 \int_{S^d}f(x)\,d\mu(x)=\int_{S^d}g(y)U_\mu(y)\,d\sigma_d(y).
\]
Conversely, every smooth function $g$ of mean zero is the negative Laplacian of a smooth function $f$. The usual dual formula for best uniform approximation therefore reduces $K(n,S^d)$ to the supremum of the absolute value of this pairing over $\mu\in\mathcal M_n$ and mean-zero $g$ bounded by one. For each fixed $\mu$, the latter supremum is the norm of $U_\mu$ in the quotient $L^1(S^d)/\C$, namely its distance to the constants. Heat regularization preserves zero mean and the $L^\infty$ bound, and converges in the weak-* topology, so restricting this supremum to smooth $g$ does not change its value. This proves~\eqref{eq:exactFavardDuality}.

Choosing $c=0$ in~\eqref{eq:exactFavardDuality} gives a sufficient upper bound in terms of $\norm{U_\mu}_1$. Since $\mu$ annihilates $\mathcal H_{<n}$, subtracting a polynomial of degree at most $n-1$ from $G_d$ leaves $U_\mu$ unchanged. The best kernel approximants found above consequently yield
\begin{equation}\label{eq:FavardBohrComparison}
 K(n,S^d)\leq H(n,S^d),\qquad d=2,3.
\end{equation}
This holds for every positive integer $n$. On $S^3$, interpolation of the Green kernel also gives an optimal linear method for each $n$. Indeed, the averaging argument of Theorem~\ref{thm:optimalLinearS2}, now over $SO(4)$, gives an operator commuting with the spherical Laplacian. Its distinct eigenvalues $-\ell(\ell+2)$ force that operator to vanish on harmonics of degree at least $n$, so its maximal error is at least $H(n,S^3)$. Convolution of $-\Delta f$ with the interpolating polynomial, with the mean of $f$ restored, attains this bound by Theorem~\ref{thm:S3fixedn}.

Nevertheless, best approximation again gives a strictly smaller constant. The obstruction to agreement of phases used on $S^2$ applies here as well.

\begin{theorem}\label{thm:strictFavard}
For each integer $n\geq2$,
\begin{equation}\label{eq:strictFavard}
 K(n,S^3)<H(n,S^3).
\end{equation}
\end{theorem}

\begin{proof}
Let $p_{n-1,3}$ be the interpolating polynomial constructed for $G_3$, and put $R_{n,3}=G_3-p_{n-1,3}$. The kernel computation gives
\begin{equation}\label{eq:residualNorm}
 \int_{S^3}|R_{n,3}(x\cdot y)|\,d\sigma_3(y)=H(n,S^3)
\end{equation}
for every pole $x$, and $\sgn R_{n,3}=\sgn q_n$ almost everywhere. The zeros of $q_n$ are simple and lie in $(-1,1)$. Annihilation of the polynomial part and~\eqref{eq:exactFavardDuality} give
\begin{equation}\label{eq:dualResidualBound}
 K(n,S^3)\leq\sup_{\mu\in\mathcal M_n}\norm{U_\mu}_1\leq H(n,S^3).
\end{equation}
Since $G_3$ is integrable, the continuous-kernel approximation used in the proof of Theorem~\ref{thm:strictFavardS2} again shows that the supremum over measures is attained.

Equality would force the rotated sign functions $\sgn R_{n,3}(x\cdot y)$ to agree up to a constant phase for almost every pole with respect to $|\mu|$. Their sign-change sets are now two-dimensional latitude spheres. For nonparallel poles, the intersection of two such latitude spheres has dimension at most one, so finitely many intersections cannot cover a latitude sphere. The same neighborhood argument as on $S^2$ therefore forces $\mu$ to be supported on one antipodal pair. Annihilation of the constant and the linear function along that axis makes both masses zero, contradicting total variation one. Thus the second inequality in~\eqref{eq:dualResidualBound} is strict, proving~\eqref{eq:strictFavard}.
\end{proof}

On both spheres the compactness argument gives a positive gap for each fixed $n\geq2$. Section~\ref{sec:explicitFavardGapS2} makes this gap explicit on $S^2$, but the exact values of $K(n,S^d)$ for $n\geq2$ remain undetermined. Neither argument establishes a gap uniform in $n$ after multiplication by $n^2$; in particular, it remains open whether $K(S^d)<H(S^d)$. The distinction between approximation of the kernel and approximation of functions is therefore already strict at each $n\geq2$, while the question about the best uniform coefficients remains.

\section*{Acknowledgments}

The author is grateful to Yuriy Kryakin for drawing his attention to these questions and for discussions of the results of this work.

In preparing this paper, the author freely used publicly available sources and results published in the open literature, computational software, including Python, and various large language models.

\end{document}